\documentclass{article}
\usepackage{amsmath,  amsfonts, amsthm, latexsym, amssymb}
\usepackage{a4wide} 
\usepackage{graphicx}
\usepackage{tikz}
\usepackage{caption}

 \numberwithin{equation}{section}

\newtheorem{theorem}{Theorem}[section]
\newtheorem{definition}[theorem]{Definition}
\newtheorem{proposition}[theorem]{Proposition}
\newtheorem{remark}[theorem]{Remark}
\newtheorem{lemma}[theorem]{Lemma}
 
\newtheorem{corollary}[theorem]{Corollary}

\usepackage[sc]{mathpazo}

\usepackage{color}

\def\eps{\varepsilon}

\newcommand{\bx}{{\bf x}}
\newcommand{\bm}{{\bf m}}

\newcommand{\ox}{{\overline x}}
\newcommand{\ux}{{\underline x}}

\newcommand{\cI}{{\cal I}}
\newcommand{\cK}{{\cal K}}
\newcommand{\cP}{{\cal P}}
\newcommand{\cW}{{\cal W}}

\newcommand{\R}{{\mathbb R}}

 \titlepage
\title{$1$-dimensional multi-agent optimal control with aggregation and distance constraints: qualitative properties and mean-field limit}
\author{ Annalisa Cesaroni and Marco Cirant}
\date{ }

\begin{document}

\maketitle
\begin{abstract}   
In this paper we consider  an optimal control problem for a  large population of interacting agents with deterministic dynamics, aggregating potential and  constraints on reciprocal distances, in dimension 1. We study  existence and qualitative properties of periodic in time optimal trajectories of the finite agents optimal control problem, with particular interest on the compactness of the solutions' support and on the saturation of the distance constraint. Moreover, we prove, through a $\Gamma$-convergence result, the consistency of the mean-field optimal control problem with density constraints with the corresponding underlying finite agent one and we deduce some qualitative results for the time periodic equilibria of the limit problem.  \medskip

\noindent
{\footnotesize \textbf{AMS-Subject Classification}}. {\footnotesize  
82C22   
91A13 
37J45 
49Q20 
35Q91 
}
\end{abstract}

\maketitle
\tableofcontents

\section{Introduction} 

The study of systems of evolving interacting agents plays a central role in the mathematical modelling of biological, social and economical phenomena. There has been in the last years an impressive development of the literature describing systems with many indistinguishable agents, and their macroscopic mean-field limits as the number of agents tends to infinity. This asymptotic analysis is typically possible whenever the effect of the population on any single agent is described by averaged quantities. There is a large literature on these models, we refer for example to the reviews \cite{carr1, carr2}, see also the introduction and the references in \cite{flos}. 
Due to the fact that self-organization is not always occurring, since the interactions between agents do rarely   lead to  global coordination or pattern formation, 
part of this literature is devoted to the analysis of centralized optimal control where a central planner tries to optimize a social cost. We refer for example to recent works  \cite{bfrs17, bptt19, hr19} on the derivation of first order optimality conditions for controlled interacting agent systems and the associated control problem in the mean field limit, that is, in the limit of dynamical systems with infinitely many agents.

This paper is concerned with this kind of centralized control. In particular we consider a deterministic system with constraints on the reciprocal distance between agents, that is
\[
\dot x^i_t = u^i_t, \qquad |x^i_t  - x^j_t | \ge c N^{-1/d}, \qquad i \neq j \in \{1, \ldots, N\}
\]
where ${\bf x} = (x^1_t, \ldots, x^N_t)$ denotes the states of the agents at time $t$, ${\bf u} = (u^1_t, \ldots, u^N_t)$ are (open-loop) controls, $c > 0$, $N$ is the number of agents and $d$ is the dimension of the ambient space, that will be typically $1$ in our results.   Controls are chosen to minimize a functional of the form
\begin{equation}\label{JTNi}
J_T^N(\bx, {\bf u}) := \frac1{2N} \sum_{i=1}^N \int_{-T/2}^{T/2}  |u^i_t|^2 dt +  \frac1N \sum_{i=1}^N \int_{-T/2}^{T/2} W(x^i_t) dt - \frac1{N^2}\sum_{i \neq j} \int_{-T/2}^{T/2} K\big(|x^i_t-x^j_t|\big) dt,
\end{equation}
where $W$ is a coercive double-well potential and $K$ is an interaction kernel. We remark that, as far as we know, centralized control problems  where aggregation and  spatial constraints between particles are enforced at the same time, as in our case,  have not yet been investigated in the literature. We point out that spatial constraints, which translates into an $L^\infty$ upper bound on the density of the population in the  mean field limit, is from the point of view of applications a quite natural assumption. We refer e.g. to the review \cite{sreview} and the references therein for a presentation and a discussion on the relevance  of  the theory of density-constrained evolutions in the Wasserstein space as a model for crowd motion.

We focus here on $T$-{\it periodic} optimal trajectories, with the aim of understanding the behavior of the system of agents in the long-time regime, and the possible formation of evolutive patterns. Since $J_T^N$ actually depends on the empirical distribution $m_{\bx_t}^N = \sum_i \delta_{x^i_t}$, namely it can be written, recalling that $|x^i_t  - x^j_t | \ge c N^{-1/d}$, as 
\[
J_T^N(\bx, {\bf u}) = \int_{-T/2}^{T/2} \frac1{2N} \sum_{i=1}^N   |u^i_t|^2 dt +   \int_{-T/2}^{T/2} \cW(m_{\bx_t}^N) dt,
\] where $\cW(m_{\bx_t}^N)= \int_{\R^d}  W(x) m_{\bx_t}^N(dx) -\int_{\{x \neq y\}} K(|x-y|) m_{\bx_t}^N(dx)m_{\bx_t}^N(dy)$,  one may hope to treat 
 the optimal control problem by passing, through $\Gamma$-convergence techniques, to  its continuous (or ``mean-field'') limit, which  is, at least formally,
\begin{equation}\label{JTi}
J_T(m, w) =\int_{-T/2}^{T/2} \int_{\R^d} \frac12\left|\frac{dw}{dt\otimes m(t,dx)}\right|^2 m(t,dx)dt +\int_{-T/2}^{T/2}\cW(m)dt,
\end{equation}
where now $m$ is a $T$-periodic flow of probability measures, $w$ is the momentum variable, and the couple $(m,w)$ satisfies the continuity equation
\[
-\partial_t m + {\rm div}(w) = 0.
\]
Note that the distance constraint $|x^i_t  - x^j_t | \ge c N^{-1/d}$ would  lead to the density constraint $0 \le m(t,x) \le \ 2^d c^{-d}\omega_d^{-1}$, where $\omega_d$ denotes the Lebsegue measure of the unit ball of $\R^d$ (see Proposition \ref{convparticelle}).  In a previous work \cite{ccbrake}, we considered the mean-field problem \eqref{JTi} when the term $\cW$ is more generally an infinite-dimensional  double-well shaped potential.  In particular, under suitable symmetry and coercivity assumptions on $\cW$,  we showed that global minimizers (steady states) of the energy $J_T$ are given by indicator functions, and  form two disjoint compact subsets $\mathcal{M}^\pm$ of $\cP_2(\R^d)$. Then, we provided the existence of non-trivial $T$-periodic critical points that exhibit an oscillatory behavior between  $\mathcal{M}^\pm$ and, as the period $T$ tends to infinity,  converge to heteroclinic connections.    We summarize these  results  in Theorem \ref{mfgex}. 

The idea of treating  finite agent optimal control problems by looking at their  mean-field approximation has been widely used in the literature, mostly in the case of stochastic optimal control, which leads to  McKean-Vlasov optimal control problems. The justification that the McKean-Vlasov optimal control problem is consistent with the limit of optimal controls for stochastic finite agent models has been provided  recently in \cite{Lacker} by probabilistic methods. For purely deterministic control problems, as the ones we consider,   the consistency of mean-field limit with the corresponding underlying finite agent control problem has been proved in the recent work \cite{flos}, see also \cite{fornasier}, by   measure-theoretical methods, in particular exploiting  the superposition principle. Nevertheless, to our knowledge there are no results under the presence of distance constraints between agents, as in the case we consider and for which we provide in this paper, in dimension $1$, the $\Gamma$-convergence result. Note that the optimization of \eqref{JTi} is also intimately connected with some problems arising in Mean Field Games (MFG). The theory of MFG, introduced in the mathematical community by Lasry and Lions \cite{LL07}, provides a framework to describe Nash equilibria of differential games with indistinguishable $N$ agents in the mean-field limit $N \to \infty$. We refer to \cite{fs} for this convergence problem in the deterministic case. The periodic in time critical points of $J_T$ we find in \cite{ccbrake} turn out to be equilibria in suitable variational (or potential) MFG models.
Note that examples of MFG exhibiting periodic solutions were found in \cite{c19, cn, ymms} (see also \cite{abc} for numerical evidences). In these models, the presence of viscosity was crucial to prove the existence of such solutions. Here, we deal with a purely deterministic problem.
Finally, we recall also some recent results in dimension $1$, providing the rigorous justification of  local and nonlocal  transport equations with nonlinear mobility as the continuous limit of deterministic follow-the-leader type models, with distance constraints between particles,  see \cite{dffr, dfr} and reference therein.
 \smallskip
 
Going back to the description of our results, the aim of the present paper is two-fold. First, to study qualitative properties of families of critical points of the $N$-agents control problem \eqref{JTNi}, under the distance constraint between agents. In particular, we are interested in the support of the optimal trajectories, and if the distance constraint is actually saturated or not. Second, to prove that the $N$-agents  control problem \eqref{JTNi} converges to the mean-field problem \eqref{JTi}, in the sense of $\Gamma$-convergence. We are able to complete this program in dimension $d = 1$. 
 As a byproduct of our $\Gamma$-convergence results, we will be able also to construct periodic solutions of the continuous problem that are evolving indicator functions of measurable sets, i.e. $\chi_{E(t)}$, $|E(t)| = 1$, and give a complete description of the evolving set $E(t)$.
\smallskip

Let us discuss the standing assumptions on $K$ and $W$. We consider a radially symmetric interaction kernel $K(|x|)$, where $K:[0,+\infty)\to [0, \infty)$ is a function such that  
\begin{equation}\label{assK}\begin{cases} r\mapsto  r^{d-1} K(r)  \in L^1_{loc}([0, +\infty), [0, +\infty)),  \\
 \text{ $K$ is nonincreasing and differentiable in $(0,+\infty)$ and 
$\lim_{r\to +\infty} K(r)=0$.}\end{cases}
\end{equation}
Moreover, we assume that $K$ is positive definite, which means  that 
\begin{equation}\label{pos} \int_{\R^d}\int_{\R^d} f(x)f(y)K(|x-y|) dxdy\geq 0\quad \text{ for all $f\in L^1(\R^d)$}\end{equation} 
and $\int_{\R^{2d}} f(x)f(y)K(|x-y|) dxdy= 0$ if and only if $f=0$. The corresponding term in \eqref{JTNi} plays the role of an aggregating term, since it is minimized whenever reciprocal distances between agents are minimized. As for the term in \eqref{JTNi} involving $W$, it models spatial preferences of agents. We assume that it has quadratic growth and it is radially increasing outside a ball, i.e.
\begin{equation}\label{assw0} \begin{cases} \text{$W\in C^1(\R^d)$ is non-negative}, \\  
 \exists C_W>0\  \text{such that} \ C_W^{-1}|x|^2-C_W\leq W(x)\leq C_W|x|^2+C_W\\ 
\exists R_0>0  \ \text{s.t.} \ \nabla W(x)\cdot x >0 \ \text{for all} \  |x|>R_0.
\end{cases} 
\end{equation}

 As we previously mentioned, we constructed in \cite{ccbrake} particular periodic solutions of the continuous problem called {\it brake orbits}, namely trajectories travelling along the same path back and forth in T/2-time. To this aim we had to assume some symmetry in the problem. Here again, we assume that $W$ is invariant under a reflection $\gamma:\R^d\to \R^d$, that is
\begin{equation}\label{ref} W(x)= W(\gamma(x))\qquad x\in\R^d.
\end{equation}
We will minimize $J^N_T$ under the following constraints:
\begin{equation}\label{kappaN}
\begin{array}{l}
\cK_T^N := \Big\{\bx_t \in W^{1,2}(\R; \R^N) :  \bx\text{ is $T$-periodic, $x^1_t<x^2_t<\dots<x^N_t$}, \
 |x^{i}_t- x^j_t |\geq \frac{1}{N}\ \forall i\neq j, \forall t\  \Big\}, \\ 
\cK_T^{N,S} :=\Big\{\bx_t \in \cK_T^N \ :  x^i_{\frac T4 + t} = x^i_{\frac T4 - t} \text{ \ and   } x^{N+1-i}_t = -x^i_{-t} \ \forall i,\ \forall t\Big\}.
\end{array} 
\end{equation}
Note that $\cK_T^N$ imposes only some ordering between agents (which is natural in dimension one) and the constraint on their reciprocal distances. In addition, $\cK_T^{N,S}$ forces symmetries that lead to non-trivial in time periodic critical points. Note that these additional constraints are {\it natural} by the symmetry assumption on $W$ (see Section \ref{soptcond}). Our first result reads as follows.

\begin{theorem}\label{thmNparticle}
Assume  \eqref{assK}, \eqref{pos}, \eqref{assw0} and
\eqref{ref}. 
\begin{enumerate} 
\item There exists $\bx\in\cK_T^N$ such that $J_T^N(\bx)=\min_{\cK_T^N}J_T^N$. Moreover every such minimizer is stationary, that is $\bx_t=\bx_s$ for all $t\neq s$ and
\[ \frac1N \sum_{i=1}^N W(x^i) dt - \frac1{N^2}\sum_{i \neq j} K\big(|x^i-x^j|\big)\leq \frac1N \sum_{i=1}^N W(y^i) dt - \frac1{N^2}\sum_{i \neq j} K\big(|y^i-y^j|\big)\] for every ${\bf y}\in(\R^d)^N$ such that $|y^{i}- y^j |\geq\frac{1}{N}$ for $i\neq j$. 
\item There exists $\bx\in\cK_T^{N, S}$  such that $J_T^N(\bx)=\min_{\cK_T^{N,S}}J_T^N$.
\item Let $d=1$, $ \bar \bx\in \cK_T^N$ be any  minimizer of $J_T^N$ constrained to $\cK_T^N$,  and $\bx\in \cK_T^{N,S} $ be any minimizer of $J_T^N$ constrained to $\cK_T^{N,S}$. Then, 
\[
|\bar x^i|, |x^i_t| \le R_0 + 1 \qquad \text{for all $t \in [0,T]$ and $i = 1, \ldots N$}
\] where $R_0$ is as in \eqref{assw0}.
\end{enumerate} \end{theorem}

Items (i) and (ii) in the previous statement are proven in a standard way arguing with the direct method. Item (iii) requires a delicate truncation procedure. Note that (iii) states that, independently on $T$ and on the number of agents, every agent is bounded to remain in a compact region of $\R$. Then, we prove that if the interaction kernel is strong enough with respect to the potential term, then agents should minimize reciprocal distances. In particular, assume that
\begin{equation}
\label{coew2} \min_{0 < r \le 2R_0 + 2} |K'(r)|  > \max_{|x| \le R_0 + 1} |\nabla W(x)|.
\end{equation}
Then, we have:
  \begin{theorem} \label{propositionsaturation} 
Besides the assumptions of Theorem \ref{thmNparticle}, suppose also that \eqref{coew2} holds. Let $ \bar \bx\in \cK_T^N$ be any  minimizer of $J_T^N$ in $\cK_T^N$ (which is stationary by Theorem \ref{thmNparticle}) and let $\bx \in \cK_T^{N,S} $ be any minimizer of $J_T^N$ in $\cK_T^{N,S}$. Then they both saturate the distance  constraint, that is \[
\bar x^{i+1} = \bar x^{i}+ \frac1{N} \quad \text{and} \quad x^{i+1}_t = x^{i}_t + \frac1{N}\qquad \text{for all $t$ and $i$.} 
\]
\end{theorem}
The previous theorem is obtained by extracting information from the optimality conditions, which typically have the form of differential inequalities by the presence of the distance constraint. To circumvent the rigidity given by such a constraint, it is convenient to write optimality conditions for groups of agents. In particular, we show that differences between barycenters
\[
\frac1{N-J} \sum_{i=J+1}^N x^i_t \quad - \quad \frac1{J} \sum_{i=1}^J x^i_t,
\]
for any $J = 1, \ldots, N-1$, describe completely reciprocal distances between agents at each time $t$, and use this fact to prove Theorem \ref{propositionsaturation}. Note that if \eqref{coew2} does not hold, the conclusion of Theorem \ref{propositionsaturation} may be false. See Remark \ref{nonsat} and the numerical experiments in Section \ref{numan}.

\smallskip

The final step in our analysis is the study of the limit $N \to \infty$. First, we define the continuous counterparts of $\cK_T^N, \cK_T^{N,S}$ as follows:
\begin{eqnarray}\nonumber  \cK_T &:=&\left\{(m,w)\ : \ m\in  C(\R, \cP_{2}^r(\R^d)), \quad m(t)  \text{ is $T$-periodic, }\right. 
\\&&  w \text{ is a  Borel $d$-vector measure  on $\R\times \R^d$, absolutely continuous w.r.t.  
 $dt \otimes m(t,dx)$, }\nonumber \\
&&    -\partial_t m + {\rm div}(w) = 0 \ \text{in the sense of distributions,   }  \nonumber \\
& &\left.  \int_{t_1}^{t_2} \int_{\R^d}\frac12 \left|\frac{dw}{dt\otimes m(t,dx)}\right|^2 m(t,dx) < \infty \ \text{for all $-\infty < t_1 < t_2 < \infty$}\right\}, \label{kappanuovo} \\
\cK_T^S &:=& \left\{ (m,w) \in \cK_T :\  m(-t) = \gamma_\# m(t),\     m\left( \frac T4 + t\right) = m\left( \frac T4 - t\right),  \forall t \in \R \right\}. \label{kappat}
\end{eqnarray} 
Then, we have a $\Gamma$-convergence results for the (constrained) functional $J_T^N$ to $J_T$, again in dimension 1.
\begin{theorem} Assume  \eqref{assK}, \eqref{pos}, \eqref{assw0}. Then, \[(J_T^N) |_{\cK_T^N} \stackrel{\Gamma}{\longrightarrow}  (J_T) |_{\cK_T} \qquad \text{as $N\to +\infty$,} \] with respect  to  convergence in $C([0,T], \mathcal{P}_2(\R))$ of the empirical measure $m_{\bx}^N(\cdot)$ to $m(\cdot)dx$  and  narrow  convergence of the  empirical measure $  \int_{-T/2}^{T/2}\delta_t\otimes w^N_\bx(t)dt$  to $w \  dt\otimes  dx$. The same result holds with the additional symmetry constraints, i.e. $(J_T^{N}) |_{\cK_T^{N,S}} \stackrel{\Gamma}{\longrightarrow}  (J_T) |_{\cK_T^S}$.
\end{theorem}
See Theorem \ref{gamma} for a more precise statement of this result. Also in this step, the key difficulty is handling the presence of the distance constraint, in particular in the $\Gamma-\limsup$ inequality, that is: given $(m,w) \in \cK_T$, to construct a sequence of competitors ${\bx}^N$ whose empirical measures converge to $m$. Note that $|x^i_t  - x^j_t | \ge c N^{-1/d}$ has to be satisfied. We are not able yet to perform this construction in dimension strictly larger than one.

As a byproduct of the $\Gamma$-convergence result, minimizers of $J_T^N$ converge to minimizers of $J_T$, and their properties pass to the limit: so, under the assumption of strong aggregation \eqref{coew2},  evenly spaced agents give in the limit a density which is indeed of the form $\chi_{E(t)}$, $|E(t)| = 1$. In particular,  as it will be stated in Corollary \ref{finalcor}, we obtain the existence of $(m_T, w_T)\in \cK_T^{S}$  minimizing $J_T$  in $\cK_T^S$ such that 
\[m_T(t)=\chi_{(a_T(t), a_T(t)+1)}\qquad w_T(t)= -\dot{a}_T(t)\chi_{(a_T(t), a_T(t)+1)}\] 
where $a_T:\R\to [-R_0-1, R_0]$  with $R_0$ as in \eqref{assw0}, is a $T$-periodic $C^2$ function minimizing 
\[\int_{-T/2}^{T/2} \frac{|\dot x_t|^2}{2}dt+\int_{-T/2}^{T/2}\int_{x_t}^{x_t+1} W(s)ds dt\] 
among $T$ periodic curves $x_t$ such that $x_{t+\frac{T}4}=x_{\frac{T}4-t}$,  and $x_t=-x_{-t}-1$.
This shows that the infinite dimensional constrained minimization problem of $J_T$ is equivalent to a one dimensional minimization problem for a single representative agent, who observes the averaged potential $x\mapsto \int_{x}^{x +1} W(s)ds$.

 \subsection*{Acknowledgements} The authors are partially supported by the Fondazione CaRiPaRo Project ``Nonlinear Partial Differential Equations: Asymptotic Problems and Mean-Field Games'', and are members of GNAMPA-INdAM. The authors wish to thank D. Gomes for having provided them with a very efficient {\it Mathematica} code, which has been used in the numerical analysis of Section \ref{numan}.
%

\section{Preliminary results}  \label{secassumption}
\subsection{Wasserstein spaces}  We briefly recall some notions for calculus in Wasserstein spaces that will be useful in the following. For a general reference on these results we refer to \cite{ags, Sbook}. First, let $ \mathcal{P}(\R^d)$ be the space of Borel probability measures on $\R^d$, endowed with the topology of narrow convergence, that is:
 \begin{definition}  Let $\mu_k, \mu\in \mathcal{P}(\R^d)$.  We say that $\mu_k\to \mu$ narrowly if 
 \[\lim_k \int_{\R^d}g(x)\mu_k(dx)= \int_{\R^d} g(x) \mu(dx)\qquad \forall g\in C_b(\R^d),\]   
where $C_b(\R^d)$ is the space of continuous and bounded functions on $\R^d$.  
 \end{definition} 
 Note that this convergence is equivalent to convergence in the sense of distributions (see \cite[Remark 5.1.6]{ags}). 
  We recall also  the notion of weak-* convergence in $L^\infty$, that is: for $\mu_k, \mu \in L^\infty(\R^d)$, $\mu_k$ is said to converge to $\mu$ weak-* in $L^\infty$ if  \[\lim_k \int_{\R^d}g(x)\mu_k(dx)= \int_{\R^d} g(x) \mu(dx)\qquad \forall g\in L^1(\R^d).\] 
\begin{definition}\label{wdef}  
Let $p\geq 1$. The Wasserstein space of Borel probability measures  with bounded $p$-moment is defined by
\[\mathcal{P}_p(\R^d)=\left\{\mu\in \mathcal{P}(\R^d)\ \Big| \int_{\R^d} |x|^p d\mu(x)<+\infty\right\}.\] 

The Wasserstein space can be endowed with the $p$-Wasserstein distance
\begin{equation}\label{wassdis} d_p(\mu, \nu)^p= \inf\left\{\int_{\R^d} \int_{\R^d} |x-y|^p d\pi(x,y)\ |\ \pi\in \pi(\mu, \nu)\right\}\end{equation} 
where $\pi(\mu,\nu)$  is the set of  Borel probability measures on $\R^d\times \R^d$ such that $\pi(A\times\R^d)=\mu(A)$ and $\pi(\R^d\times A)=\nu(A)$ for  any Borel set $A\subseteq \R^d$.  

We introduce a subspace of regular measures as follows
\[
  \cP_{p}^{r} = \{m \in \cP_p(\R^d) \ : \ \exists \ 0 \le \tilde m \le 1 \ \text{ a.e. on $\R^d$ s.t. $m = \tilde m dx$}\} .
 \]
 \end{definition}

 Note that $\mathcal{P}_p(\R^d)\subset \mathcal{P}_q(\R^d)$ for $p<q$, and by Jensen inequality, $d_p(\mu, \nu)\leq d_q(\mu, \nu)$ for $p<q$.  We then recall the following result about narrow convergence and convergence in Wasserstein spaces. 
\begin{lemma}\label{equiconv}
  $\mathcal{P}_p(\R^d)$ endowed
with the p-Wasserstein distance is a separable complete  metric space.  A set $\mathcal{M}\subseteq \mathcal{P}_p(\R^d)$  is relatively compact if and only if it has uniformly integrable  $p$-moments, that is 
 \[\lim_{R\to +\infty} \sup_{\mu\in \mathcal{M}}\int_{\R^d\setminus B(0, R)} |x|^pd\mu(x)=0.\] 
Let  now $\mu_k,\mu\in \mathcal{P}_p(\R^d)$ for some $p\geq 1$. Then the statements below are equivalent:
\begin{enumerate}
\item $d_p(\mu_k, \mu)\to 0$
\item $\mu_k\to \mu$ narrowly and  $\mu_k$ have uniformly integrable $p$-moments.
\end{enumerate} 

Finally, for any $\nu\in \cP_p(\R^d)$, the map $\mu\to d_p(\mu, \nu)$ is lower semicontinuous with respect to narrow convergence. 

\end{lemma}
\begin{proof} We refer to  \cite[Prop. 7.1.5]{ags}. Note that if $\mathcal{M}$ has uniformly integrable $p$-moments then it is tight, i.e. for all $\eps>0$ there exists 
 $K_\eps\subseteq \R^d$  compact for which  $\sup_{\mu\in \mathcal{M}}\int_{\R^d\setminus K_\eps}d\mu(x)\leq \eps$. 
 
 The lower semicontinuity of the Wasserstein distance is proved in \cite[Proposition 7.1.3]{ags}. 
\end{proof}
\begin{remark}\label{remequiconv} \upshape 
Note that, if for some $q>p$, 
 \[  \sup_{\mu\in \mathcal{M}}\int_{\R^d} |x|^q d\mu(x)<+\infty \] 
 then $\mathcal{M}$ has uniformly integrable $p$-moments. 
\end{remark} 

The following  semicontinuity and continuity property with respect to narrow convergence and convergence in $\cP_2(\R^d)$ respectively will be useful.
\begin{lemma}\label{lscW} 
Let   $\mu_k, \mu$ be Borel probability measures on $\R^d$ such that  $\mu_k\to \mu$ narrowly. Then
\[\liminf_k \int_{\R^d} W(x)\mu_k(dx)\geq \int_{\R^d} W(x)\mu(dx)\] for all lower semicontinuous functions $W$. 

If moreover   $W$ is as in \eqref{assw0} and  $\mu_k, \mu\in \mathcal{P}_2(\R^d)$ then 
\[\lim_k d_2(\mu_k,\mu)=0\quad \text{ if and only if  }\quad \lim_k\int_{\R^d} W(x) \mu_k(dx)= \int_{\R^d} W(x) \mu(dx).\] 
\end{lemma}
\begin{proof} For the proof we refer to  \cite[Lemma 5.1.7]{ags} and \cite[Lemma 5.7]{ccbrake}. 
\end{proof}

  Finally,  we recall some results that will be useful in the following about the functional 
  \[(\mu, w)\to \int_{\R^d}\frac12\left|\frac{dw}{ d\mu}\right|^2 \mu(dx) \] defined on couples $(\mu, w)$ such that $\mu\in \cP_1(\R^d)$ and $w$ is  Borel $d$-vector measure  on $ \R^d$, absolutely continuous w.r.t.   $\mu(dx)$.  
 First of  all observe that by H\"older inequality the total variation of $|w|$ satisfies
 \begin{equation}\label{tv} |w|(\R^d)\leq \left( \int_{\R^d}\left|\frac{dw}{ d\mu}\right|^2 \mu(dx)\right)^{1/2}\end{equation} and moreover the functional is (joint) lower semicontinuous with respect to narrow convergence of measures.
 \begin{lemma}\label{lemmalsc} Let $\mu_k, \mu\in \mathcal{P}(\R^d)$, with $\mu_k\to \mu$ narrowly. Let $w_k, w\in \mathcal{P}(\R^d, \R^d)$ be Borel vector measures such that $w_k\to w$ narrowly and $w_k, w$ are  absolutely continuous with respect to 
 $dt \otimes \mu_k(t,dx)$ and $dt \otimes \mu(t,dx)$, respectively. Then
 \[\liminf_k \int_{\R^d}\left|\frac{dw_k}{d \mu_k}\right|^2 d\mu_k\geq  \int_{\R^d}\left|\frac{dw}{ d\mu}\right|^2 d\mu. \]
 \end{lemma}
 \begin{proof}
The result is proved in   \cite[Lemma 9.4.3]{ags}. 
 \end{proof} 
 The  evolutive version of the previous functional is given by 
  \[(\mu, w)\to \int_{-T/2}^{T/2}\int_{\R^d}\frac12\left|\frac{dw}{dt\otimes \mu(t, dx)}\right|^2 \mu(t,dx)dt\]
  defined on couples $(\mu, w)$ such that $\mu \in  C(\R, \cP_1(\R^d))$, 
$ w$  is a  Borel $d$-vector measure  on $\R\times \R^d$, absolutely continuous w.r.t.  
 $dt \otimes \mu(t,dx)$. 
 We recall the following uniform continuity property. 
 \begin{proposition} Let  $(\mu, w)$ as before and assume that   $ -\partial_t m + {\rm div}(w) = 0$ in the sense of distributions.  Then 
\begin{equation}\label{unmezzoh}
d_2^2(\mu(t), \mu(s)) \le |t-s|  \int_{t_1}^{t_2} \int_{\R^d}\frac12\left|\frac{dw}{dt\otimes \mu(\tau,dx)}\right|^2 \mu(\tau,dx)d\tau, \qquad \forall t,s \in (t_1, t_2).
\end{equation} 
 \end{proposition} 
 \begin{proof}  This result is proved in  \cite[Theorem 8.3.1]{ags}. 
 \end{proof} 

\subsection{Systems with $N$ agents}
Let $\bx = (x^1, \ldots, x^N) \in (\R^d)^N$, which represents the states of $N$ agents.  We  associate to 
the vector $\bx$ the empirical measure 
\begin{equation}\label{empirical}
m^N_\bx = \frac1N\sum_{i=1}^N \delta_{x^i}.
\end{equation} 

 We have the following convergence result. 
\begin{proposition}\label{convparticelle}  
Let $\bx \in (\R^d)^N$. Suppose that $m^N_\bx$ converges narrowly to some $\mu$  as $N \to \infty$ and 
 that for some $c > 0$  not depending on $N$ there holds   \begin{equation}\label{mindist}
|x^i -x^j| \ge \frac c{N^{1/d}} \qquad \text{for all $i \neq j$}.
\end{equation}
  Then   $\mu$ has a density $m\in L^\infty(\R^d)$, and $\|m\|_{\infty}\leq 2^d c^{-d}\omega_d^{-1}$.
 
\end{proposition} 
\begin{proof}  Note that by narrow convergence $\int_{\R^d} \mu(dx)=1$. We prove that $\mu$ has a density. Let  \begin{equation}\label{defde}\delta: = \frac c{N^{1/d}}\end{equation}  and $\phi\in C_b(\R^d)$, globally Lipschitz continuous and with $\phi\geq 0$. Observe that
for every $N$, using the bound \eqref{mindist} and  \eqref{defde}, 
\begin{eqnarray*} \int_{\R^d}\phi(y)m^N_\bx(dy)& =&\sum_i \frac{\phi(x^i)}{N}=\sum_i \frac{2^d}{\omega_d c^d}\int_{B(x^i, \delta/2)}[\phi(x^i)-\phi(y)+\phi(y)] dy\\ &\leq &
\frac{2^d}{c^d\omega_d}  \sum_i \int_{B(x^i, \delta/2)} \phi(y) dy+  \frac{2^d}{c^d\omega_d} \|\nabla \phi\|_\infty  \frac{\delta}{2} \omega_d \frac{\delta^d}{2^d} N\\ &\leq& \frac{2^d}{c^d\omega_d}\int_{\R^d}\phi(y)dy+\frac{c\|\nabla \phi\|_\infty}{2N^{1/d}}. 
\end{eqnarray*} 
By narrow convergence we conclude that $\int_{\R^d} \phi(y)d\mu(y)\leq  \frac{2^d}{c^d\omega_d}\int_{\R^d}\phi(y)dy$ for any $\phi\in C_b(\R^d)$, with $\phi\geq 0$ and globally Lipschitz.  So, for every measurable set $A$ of finite measure, and every nonnegative smooth mollifier $\rho$ we get that $\int_{\R^d} \chi_A\star \rho (y)  d\mu(y)\leq \frac{2^d }{c^d\omega_d}\int_{\R^d} \chi_A\star \rho (y)  dy$, which implies $ \mu(A)\leq \frac{2^d }{c^d\omega_d}  |A|$.
 This in particular gives that $\mu$ has a density $m$ and that $\|m\|_\infty\leq \frac{2^d}{\omega_d c^d}$.
\end{proof} 

\begin{remark}\label{rem1dconvparticelle} \upshape In dimension $d=1$, we may assume without loss of  generality that $x^1\leq x^2\leq\dots\leq x^N$. Observe that given  $\bx \in \R^N$ as in Proposition \ref{convparticelle} which satisfy in addition $|x^i-x^{i+1}|= \frac{c}{N}$ for all $i=1, \dots, N-1$, then there exists $a \in \R$ such that $\mu=c^{-1}\chi_{(a,a+c)}$. 

Indeed, using that $x^{i+1}-x^i=\frac{c}{N}$ we get that  $\text{supp} \ m_{\bx}^N= \left[x^1, x^1+\frac{c(N-1)}{N}\right]\subseteq [x^1, x^1+c]$ for all $N$.  Therefore since by assumption $m^N_\bx$ converges narrowly to some $\mu$, we get that, eventually passing  to a subsequence, $x^1\to a$ as $N\to +\infty$ and $\text{supp} \mu \subseteq [a, a+c]$. Finally, since $\mu(dx)=m(x)dx$ and $0\leq m(x)\leq c^{-1}$, then necessarily $\mu=c^{-1}\chi_{(a, a+c)}$. 
\upshape 
\end{remark} 
If $\bx_t \in W^{1,2}_{\rm per}((-T/2,T/2); (\R^d)^N)$, we associate to it the curve of probability measures (representing the dynamic state of $N$ agents)
\begin{equation}\label{empirical2}
m^N_{\bx}(t) = \frac1N\sum_{i=1}^N \delta_{x^i_t} 
\end{equation}  and to $\dot \bx_t$ the momentum measures 
\begin{equation}\label{empiricalspeed}
 w^N_\bx(t):= \frac1N\sum_{i=1}^N\dot x_t^i \delta_{x^i_t}, \qquad \widetilde w^N_\bx := \int_{-T/2}^{T/2}\delta_t \otimes w^N_\bx(t) dt \in \mathcal M([-T/2,T/2] \times \R).
\end{equation} 
Note that $\tilde w^N_\bx(t)$ is absolutely continuous with respect to $dt\otimes m_{\bx}^N(t)$, with density $\frac{dw^N_\bx(t)}{dm_{\bx}^N(t)}=\sum_{i=1}^N\dot x_t^i \chi_{x=x^i_t}$. Moreover,
\begin{equation}\label{form1} \int_\R \left|\frac{dw^N_\bx(t)}{dm_{\bx}^N(t)}\right|^2 m_{\bx}^N(t)(dx) = \frac{1}{N}\sum_{i=1}^N  |\dot{x}^i_t|^2 \qquad \text{for a.e. $t \in (-T/2,T/2)$}, 
\end{equation}
see e.g. \cite[Lemma 6.2]{flos}. 
\begin{remark}\label{remreg} \upshape Note that in particular, by \eqref{unmezzoh} and \eqref{form1}, we get that if $\frac{1}{N}\sum_{i=1}^N\int_{-T/2}^{T/2}  |\dot{x}^i_t|^2dt \leq C $, then
$d_2^2(m^N_\bx(t), m^N_\bx(s)) \leq C (t-s)$ for all $-T/2\leq s<t\leq T/2$. \end{remark} 
 
\subsection{The interaction energy} 

We define the energy interaction functional for $m\in L^1(\R^d)$ 
\begin{equation}\label{int} \mathcal{I}(m)=\int_{\R^d}\int_{\R^d} m(x)m(y) K(|x-y|)dxdy =\int_{\R^d} m(x)V_m(x)dx
\end{equation} 
 where $V_m$ is the interaction potential 
\begin{equation}\label{pot} V_m(x)=m*K(x)= \int_{\R^d}  m(y) K(|x-y|)dy.
\end{equation}  It is well known that if $m\in L^1(\R^d)\cap L^\infty(\R^d)$ then $V_m\in C(\R^d)$ and $\lim_{|x|\to +\infty} V_m(x)=0$ (see \cite[Lemma 5.2]{ccbrake}).
%

With a slight abuse of notations, $\cI$ can be evaluated at $m^N_\bx= \frac1N\sum_{i=1}^N \delta_{x^i}$ if $x^i\neq x^j$ for $i\neq j$,  in the sense that
\begin{equation}\label{equivI}
\cI(m^N_\bx) = \iint_{\R^{2d} \setminus \Delta} K(|x-y|) m^N_\bx(dx) m^N_\bx(dy) = \frac1{N^2} \sum_{  i \neq j } K(|x^i-x^j|),
\end{equation}
where $\Delta = \{(x,y) \in \R^d \times \R^d : x = y \}$.

We will use the following continuity property of the interaction functional. 
\begin{proposition}\label{proplimite} 
Suppose that $\bx \in (\R^d)^N$ satisfies \eqref{mindist} for some $c$ not depending on $N$. 
 
Then,
\[
\cI(m^N_\bx) \le C_{K} \]
for some positive $C_K$ that depends on $c, d, K$ (and not on $\bx, N$).

Assume moreover that $m^N_\bx$ converges narrowly to some $\mu$  as $N \to \infty$.  Then \[\lim_N \cI(m^N_\bx)= \cI(m)\] where    $m\in L^\infty(\R^d)$ is the density of $\mu$, given  by Proposition \ref{convparticelle}.  \end{proposition}
\begin{proof} Let $\delta$ as in \eqref{defde} be  the bound from below on the distance between $x^i, x^j$. We first deduce a preliminary estimate on $\sum_{j \neq i : x^j \in B_r(x^i)} K(|x^i-x^j|)$, for any fixed $i = 1, \ldots, N$ and $r \ge \delta$. To this aim, let $\hat x^j := x^j - \frac\delta8 \frac{x^j-x^i}{|x^j-x^i|}$. Note that
\begin{align}
& B_{ \frac\delta8}(\hat x^j) \subset B_r(x^i) \qquad \forall j \text{ such that } x^j \in B_r(x^i), \notag \\
& \text{for all $y \in B_{ \frac\delta8}(\hat x^j)$, \qquad $|y - x^i| \le |x^j-x^i|$,} \label{3properties} \\
& B_{ \frac\delta8}(\hat x^j) \cap B_{ \frac\delta8}(\hat x^k) = \emptyset \qquad \forall j \neq k. \notag
\end{align}
Indeed, since $|x^j - x^i| \ge  \delta$, by the definition of $\hat x^j$ we get
\[
|\hat x^j - x^i| \le  \Big| x^j - x^i -  \frac\delta8 \frac{x^j-x^i}{|x^j-x^i|}\Big| =  \Big| |x^j - x^i| -  \frac\delta8\Big| = |x^j - x^i| -  \frac\delta8 \le r,
\]
that yields the first claim. Then, by the triangle inequality and the previous line,
\[
|y - x^i| \le |y - \hat x^j| + |\hat x^j - x^i| \le  \frac\delta8 + |x^j - x^i| -  \frac\delta8 = |x^j - x^i|,
\]
Finally, since $x^j \in B_{ \frac\delta8}(\hat x^j)$ for all $j$,
\[
\delta \le |x^j - x^k| \le |x^j - \hat x^j| + |\hat x^j - \hat x^k| + |\hat x^k - x^k| \le |\hat x^j - \hat x^k| + \frac \delta4,
\]
hence $ |\hat x^j - \hat x^k| \ge \frac34 \delta$, and $B_{ \frac\delta8}(\hat x^j), B_{ \frac\delta8}(\hat x^k)$ must be disjoint.

Therefore, using the three properties \eqref{3properties}, and the assumption that $K(\cdot)$ is non-increasing, we get
\begin{multline}\label{Kineq}
\sum_{j \neq i : x^j \in B_r(x^i)} K(|x^i-x^j|) = \sum_{j \neq i : x^j \in B_r(x^i)} \frac{1}{|B_{ \frac\delta8}(\hat x^j)|}\int_{B_{ \frac\delta8}(\hat x^j)} K(|x^i-x^j|) dy \\  \le  \sum_{j \neq i : x^j \in B_r(x^i)} \frac{1}{|B_{ \frac\delta8}(0)|}\int_{B_{ \frac\delta8}(\hat x^j)} K(|y-x^i|) dy = \frac{1}{|B_{ \frac\delta8}(0)|} \int_{\bigcup_{j} B_{ \frac\delta8}(\hat x^j)} K(|y-x^i|) dy \\
\le \frac{1}{|B_{ \frac\delta8}(0)|} \int_{B_r(x^i)} K(|y-x^i|) dy = \frac{8^d N}{c^d  \omega_d} \int_{B_r(0)} K(|y|) dy.
\end{multline}
A first consequence of this inequality is that
\begin{multline*}
\sum_{j \neq i} K(|x^i-x^j|) = \sum_{j \neq i : x^j \in B_1(x^i)} K(|x^i-x^j|) + \sum_{j : x^j \in \R^d \setminus B_1(x^i)} K(|x^i-x^j|) \\
\le N \left( \frac{8^d}{c^d \omega_d} \int_{B_1(0)} K(|y|) dy + K(1)\right),
\end{multline*}
that provides the first stated bound on $\cI(m^N_\bx)$
\[
\cI(m^N_\bx) = \frac1{N} \sum_{  i  }\left[\frac1{N} \sum_{  j \neq i  } K(|x^i-x^j|)\right] \le  \left( \frac{8^d}{c^d \omega_d} \int_{B_1(0)} K(|y|) dy + K(1)\right).
\]
\medskip

We conclude showing the continuity of the interaction energy. For any $r > 0$, define the truncated interaction kernel
\[
K_r(t) :=
\begin{cases}
K(t) & \text{if $t \ge r$,} \\
K(r) & \text{if $0 \le t \le r$}.
\end{cases}
\]
Since $K_r(|\cdot_x - \cdot_y|) \in C(\R^{2d}) \cap L^\infty(\R^{2d})$, it is standard (see for example \cite[Proposition 7.2]{Sbook}) that
\begin{equation}\label{conv1}
 \iint_{\R^{2d}} K_r(|x-y|) m^N_\bx(dx) m^N_\bx(dy) \to \iint_{\R^{2d}} K_r(|x-y|) m(x) m(y) dxdy \qquad \text{as $N \to \infty$}.
\end{equation}
Moreover,
\begin{multline*}
\cI(m^N_\bx) = \frac1{N^2} \sum_{  i \neq j } \big(K-K_r\big)(|x^i-x^j|)  + \frac1{N^2} \sum_{  i \neq j } K_r(|x^i-x^j|) \\ =  \frac1{N} \sum_{  i  }\left[ \sum_{j \neq i : x^j \in B_r(x^i)} \big(K-K_r\big)(|x^i-x^j|) \right] +  \iint_{\R^{2d}} K_r(|x-y|) m^N_\bx(dx) m^N_\bx(dy)
\end{multline*}
Therefore, using \eqref{Kineq},
\begin{equation}\label{conv2}
\Big| \cI(m^N_\bx) -  \iint_{\R^{2d}} K_r(|x-y|) m^N_\bx(dx) m^N_\bx(dy) \Big| \le C_1 \int_{B_r(0)} K(|y|) dy
\end{equation}
for some $C_1 > 0$ not depending on $N, r$. Note also that since $\|m\|_1 = 1$,
\begin{multline}\label{conv34}
\Big| \cI(m) - \iint_{\R^{2d}} K_r(|x-y|) m(x) m(y) dxdy \Big| = \\ \iint_{|x-y| \le r} \big(K-K_r\big)(|x-y|) m(x) m(y) dxdy \le \int_{\R^d} \int_{B_r(0)} K(|z|) m(x+z) dz \, m(x) dx \\ \le \|m\|_{\infty} \int_{B_r(0)} K(|z|) dz.
\end{multline}
Hence, putting together \eqref{conv1}, \eqref{conv2} and \eqref{conv34} we obtain
\[
\limsup_{N\to \infty} \Big| \cI(m^N_\bx) -  \cI(m) \Big| \le (C_1 +  \|m\|_{\infty})\int_{B_r(0)} K(|y|) dy,
\]
and since $K(|\cdot|)$ is locally in $L^1$, we obtain the convergence assertion by letting $r \to 0$.
\end{proof}

\subsection{Periodic solutions of the continuous problem} \label{mfg} 

We recall in this section some results obtained in \cite{ccbrake} for the mean-field minimization problem over $\cK_T$ and $\cK_T^S$ (defined in \eqref{kappanuovo} and \eqref{kappat})
\begin{equation}\label{jt}
J_T(m,w) =\int_{-T/2}^{T/2}\int_{\R^d} \frac12\left|\frac{dw}{dt\otimes m(t,dx)}\right|^2 m(t,dx)dt +\int_{-T/2}^{T/2} \cW(m)dt,
\end{equation} 
 which is the continuous counterpart of \eqref{JTNi}. The potential part of $J_T$ is given by
 \begin{equation}\label{stat} m\mapsto \mathcal{W}(m):= \int_{\R^d} W(x)m(dx)-  \mathcal{I}(m), \qquad m \in  \cP_{2}^r(\R^d).
 \end{equation} 
 
In the following, we  will identify, with a slight abuse of notation, $m$ and $w$  with their  densities. 
 Note that $w\in L^2([-L,L]\times \R^d)$ for all $L>0$, due to the fact that 
\[\int_{t_1}^{t_2}\int_{\R^d} |w|^2dxdt\leq  \int_{t_1}^{t_2} \int_{\R^d} \left| \frac{w(t,x)}{m(t,x)}\right|^2 m(t,x) \, dx dt. \]
Moreover, we recall that elements of $\cK$ enjoy the   uniform continuity property \eqref{unmezzoh}. 

We recall that minimizers of \eqref{stat} are stationary minimizers (equilibria) of
the energy $J_T$  over $\cK_T$, but we are interested in particular in constructing non-trivial evolutive critical points. These will be found as minimizers of $J_T$ over $\cK_T^S$, which are time periodic and oscillate between stationary solutions. Moreover, sending the period to infinity, it is possible to obtain equilibria which are defined for all times and connect the stationary points.  Existence of such families of minimizers have been proved  in \cite[Theorem 1.1, Theorem 1.2, Theorem 1.3]{ccbrake}. We collect here the main results. Note first that symmetry constraints in $\cK^S_T, \cK^S$ are natural, in a sense which will be elaborated in Remark \ref{remsimm}. Moreover, note  also that to guarantee that minima of $\cW$ are made up of two disjoint subsets $\mathcal{M}^\pm$ of $\cP_{2}^r(\R^d)$, we will ask $W$ to have two sufficiently large plateaus (see the following assumption \eqref{assw1}). Such an assumption is not really needed to prove other results for the $N$-agents system.
\begin{theorem}\label{mfgex} Assume  \eqref{assK}, \eqref{pos}, \eqref{assw0},  \eqref{ref}, \eqref{coew2}, and in addition that
\begin{equation}\label{assw1}
\begin{cases}
 \exists a^+, a^-\in \R^d, r_0 > 0 \text{ such that  }B(a^+, r_0) \cap B(a^-, r_0) = \emptyset, \quad |B(a^{\pm},r_0)| \ge 1 \\
  \qquad { and } \quad W(x)=0 \Leftrightarrow x\in B(a^+, r_0)\cup B(a^-, r_0).\end{cases}
\end{equation}
\begin{enumerate}
\item  There holds  $\min_{ \cP_{2}^r(\R^d) } \cW=-\mathcal{I}(\chi_{B_{r}})$ for $r=(\omega_d)^{-1/d}$,  and   all the minimizers of \eqref{stat} are given by $\mathcal{M}^+\cup\mathcal{M}^-$, where $\mathcal{M}^-=\gamma_\#\mathcal{M}^+$ and 
\[\mathcal{M}^+=\{\chi_{E},  \text{ where } E= B(x', (\omega_d)^{-1/d})\subseteq B(a^+,r_0) \text{ for some $x'\in\R^d$}\}. \]
If $|B(a^{\pm},r_0)|=1 $, then $\mathcal{M}^\pm=\{\chi_{B(a^\pm, r)}\}$. 
 $\mathcal{M}^+$ and $ \mathcal{M}^-$ are compact subsets of $\mathcal{P}_2(\R^d)$, and $d_2(\mathcal{M^+}, \mathcal{M^-})>0$. We define $q_0:= d_2(\mathcal{M^+}, \mathcal{M^-})/2$. 

\item Let $q \in (0,  q_0)$, where $q_0$ is defined in the previous item. Then  there exists $\bar T=\bar T(q) > 4$   such that, for any $T \ge \bar T$, there exists a $T$-periodic minimizer $(m^T, w^T)\in\cK_T^{S} $ of  the functional $J_T$ in $\cK_T^S$, which satisfies
\[
\begin{cases}
d_2(m^T(t), \mathcal{M}^+) < q & \forall t \in \left(s, \frac T2- s\right) \\
d_2(m^T(t),\mathcal{M}^-) < q & \forall t \in \left(-\frac T2 + s,  - s\right),
\end{cases}
\]for some $0<s  < C = C(q)$ (not depending  on $T$). 
Moreover $\bar T(q)\to +\infty$ as $q\to 0$ and 
%
\[\lim_{T\to +\infty}d_2^2\left( m^T\left( \frac{T}{4}\right), \mathcal{M}^+\right)=0=\lim_{T\to +\infty}d_2^2\left( m^T\left( -\frac{T}{4}\right),\mathcal{M}^-\right). \] 
  \end{enumerate} 
\end{theorem}
We end up with the following useful observation. 
\begin{remark}\label{remsimm} \upshape By the convexity of the function $(m,w)\mapsto \frac{|w|^2}{m}$ and the regularity of the potential functional $\mathcal{W}$  it is possible to show that if $(\bar m, \bar w)\in \cK_T^S$ is a minimizer  of $J_T$ in $\cK_T^{S}$, then for all $(m, w) \in \cK_T^{S}$
\begin{multline}\label{nasheq}
\int_{-T/2}^{T/2} \int_{\R^d}  \left| \frac{\bar w(t,x)}{\bar m(t,x)}\right|^2 \bar m(t,x) + (W(x)-V_{ \bar m}(x))\bar m(t,x) dx dt  \\ \leq \int_{-T/2}^{T/2} \int_{\R^d}  \left| \frac{w(t,x)}{m(t,x)}\right|^2 m(t,x) +  (W(x)-V_{ \bar m}(x)) m(t,x) dx dt \end{multline} 
where $V_{\bar m}$ is defined in \eqref{pot}. The argument is detailed in  \cite{BC16}, see also \cite{notebari} and is based on the idea of computing the functional $J_T$ on $(\lambda \bar m+(1-\lambda) m), \lambda \bar w+(1-\lambda)w)$, use convexity of the kinetic part and regularity of the potential part and then send $\lambda\to 0$. Such a minimality of $(\bar m, \bar w)$ can be regarded as a mean field Nash equilibrium property.

Moreover, it is possible to show, using the symmetry assumption \eqref{ref}, that actually   the minimization property \eqref{nasheq}  holds for {\it all} competitors  $(m, w) \in \cK_T$. For the details we refer to \cite[Remark 3.4]{ccbrake}.  Therefore, following  \cite{cas16}, the fact that $(\bar m, \bar w)$ satisfies  \eqref{nasheq} for all $(m, w) \in \cK_T$ could   be used as a starting point to derive optimality conditions, where an additional ``pressure'' terms and an ergodic constant will appear   due to density constraints and $T$-periodicity, whereas  no further multipliers related to $m(T/4 + t) = m(T/4  -t)$, $m(-t) = \gamma_\# m(t)$ appear in view of the symmetry assumption \eqref{ref}.
\end{remark}

 \section{Periodic solutions of the $N$-agents system in dimension $1$}\label{sectionproperties}  
 From now on, we will restrict our study to dimension $d=1$. We will analyze some qualitative properties of the optimal trajectories of the $N$-agents system. That is, we will consider periodic minimizers, with and without symmetries, of $J_T^N$.
 

We first recall the form of $J_T^N$ (which is the discrete counterpart of $J_T$):
\begin{equation}\label{jtn}
J_T^N(\bx) = \frac1N \sum_{i=1}^N \int_{-T/2}^{T/2} \frac{|\dot{x}^i_t|^2 }{2} dt +  \frac1N \sum_{i=1}^N \int_{-T/2}^{T/2} W(x^i_t) dt - \frac1{N^2}\sum_{i \neq j} \int_{-T/2}^{T/2} K\big(|x^i_t-x^j_t|\big) dt,
\end{equation}
where $\bx \in W^{1,2} (\R; \R^N)$ and is $T$-periodic. Note that $J_T^N(\bx) < +\infty$ only if  $x^i_t \neq x^j_t$ for all $i \neq j$ and a.e. $t$. 
Recalling \eqref{form1},  the functional $J_T^N(\bx)$ can be equivalently written as (see also \eqref{equivI}),
\begin{multline*}
J_T^N(\bx) =  J_T(m^N_\bx,w^N_\bx) = \int_{-T/2}^{T/2}\int_{\R} \frac12 \left|\frac{dw^N_\bx(t)}{dm_{\bx}^N(t)}\right|^2 m_{\bx}^N(t)(dx)dt+  \int_{-T/2}^{T/2} \int_{\R} W(x)m^N_\bx(t)(dx) dt \\ - \int_{-T/2}^{T/2}\int_{\R^{2} \setminus \Delta} K\big(|x-y|\big) m^N_\bx(t)(dx) m^N_\bx(t)(dy) dt.\end{multline*} 
This functional will be considered over $\cK^N_T$ and $\cK^{N,S}_T$ (defined in \eqref{kappaN}). Since we have set our problem on the real line, it is convenient to fix the reflection as follows: $\gamma(x)=-x$ for all $x \in \R$.

First of all we have  a straightforward result about existence of minimizers in $\cK_T^N$ and $\cK_T^{N,S}$.

\begin{proof}[Proof of Theorem \ref{thmNparticle}, (i) and (ii)] The proof is based on the classical direct method, so we briefly sketch its argument. 
 First of all we observe that for every $\bx\in \cK_T^N$, by  assumption \eqref{assw0}  and by Proposition \ref{proplimite}, there exists a constant $C_K$ depending on the interaction kernel $K$ and on the dimension $d$, but not on $N$ and $\bx$ such that 
\[ J_T^N(\bx)  \geq    \frac1N \sum_{i=1}^N \int_{-T/2}^{T/2}  \frac{|\dot{x}^i_t|^2}{2} dt - C_WT +  \frac{C_W^{-1}}{N} \sum_{i=1}^N \int_{-T/2}^{T/2}|x^i_t|^2 dt-C_KT\geq -C_WT-C_KT\]
where $C_W$ is as in \eqref{assw0}.   

We prove just  item (ii), since the proof of item (i) is completely analogous (and easier). 
 Let $\eta=\inf_{\cK_T^{N,S}} J_T^N$. We fix a minimizing sequence $\bx_n$.  Since $J_T^N(\bx_n)\leq \eta+1$, we get by the above inequality   that  $\frac1{2N} \sum_{i=1}^N \int_{-T/2}^{T/2} |(\dot{x}_n^i)_t|^2 dt  \leq \eta+1+C_KT+C_WT$.   Again, by the same inequality as above we get  
  $ \frac{1}{N} \sum_{i=1}^N \int_{-T/2}^{T/2}|(x_n^i)_t|^2 dt \leq C_W(\eta+1+C_WT+C_KT)$.  Therefore we end up with  $\|\bx_n\|_{W^{1,2}_{\text{per}}(\R, \R^N)} \leq C$, where $C$ is independent of $n$.  Therefore by Sobolev embedding $\bx_n$ are equi-H\"older continuous, and $\bx_n\to \bx$  (up to a subsequence) weakly in $W^{1,2}_{\text{per}}(\R, \R^N)$ and also uniformly. We conclude by uniform convergence that $\bx\in \cK_T^{N,S}$. Moreover $\bx$ is a minimizer by  weak lower semicontinuity of $\frac{1}{2N} \sum_{i=1}^N \int_{-T/2}^{T/2}|(\dot{x}_n^i)_t|^2 dt$ and  by continuity of $ \frac1N \sum_{i=1}^N \int_{-T/2}^{T/2} W(x^i_t) dt - \frac1{N^2}\sum_{i \neq j} \int_{-T/2}^{T/2} K\big(|x^i_t-x^j_t|\big) dt$.

\end{proof}

 \subsection{Compact support}
We now show  that periodic minimal  solutions of the $N$-agents system  have a compact support,  which is independent of the number of agents 

\begin{proof}[Proof of Theorem \ref{thmNparticle}, (iii)] We prove the result only for an $\bx\in \cK_T^{N,S} $ which minimizes $J_T^N$ in $\cK_T^{N,S}$, since the other case is completely analogous (and even easier). It is enough to prove that $\min_{t \in [0,T]} x^1_{t} \ge -R_0 - 1$: this would indeed yield for all $i, t$,
\[ 
-R_0 - 1 \le \min_{t \in [0,T]} x^1_{t} \le x^1_{t} < x^i_{t} < x^N_{t} = -x^1_{-t} \le -\min_{t \in [0,T]} x^1_{t} \le R_0 + 1.
\] 
Suppose by contradiction that $x^1_{\hat t} < -R_0 - 1$ for some $\hat t \in [0, T]$. Let $\bar x_t^i$ be defined as follows (see Figure \ref{trunc}):

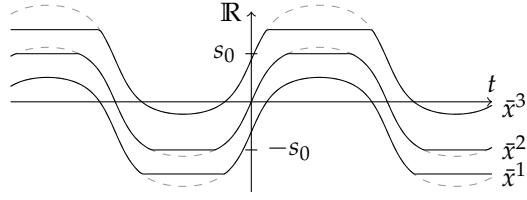
\begin{figure}
\centering
\begin{tikzpicture}[scale=0.80]
      \draw[->] (-4,0) -- (4,0) node[above] {$t$};
      \draw[->] (0,-1.5) -- (0,1.5) node[left] {$\R$};
      
    \draw[scale=1,domain=-4:4,smooth,samples=50,variable=\x, dashed, opacity=.4] plot ({\x},{atan(2*sin(500* \x /6.28))/70 - .5}) ;
        \draw[scale=1,domain=-4:4,smooth,samples=50,variable=\x, dashed, opacity=.4] plot ({\x},{atan(2*sin(500* \x /6.28))/70 }) ;
    \draw[scale=1,domain=-4:4,smooth,samples=50,variable=\x, dashed, opacity=.4, ] plot ({\x},{atan(2*sin(500* \x /6.28))/70 + .7}) ;

\draw[scale=1,domain=-4:4,smooth,samples=150,variable=\x] plot ({\x},{max(atan(2*sin(500* \x /6.28))/70 - .5,-1.2)}) node[right] {$\bar x^1$};
        \draw[scale=1,domain=-4:4,smooth,samples=150,variable=\x] plot ({\x},{max(-.8, min(atan(2*sin(500* \x /6.28))/70, .8)) }) node[right] {$\bar x^2$};
    \draw[scale=1,domain=-4:4,smooth,samples=150,variable=\x] plot ({\x},{min(atan(2*sin(500* \x /6.28))/70 + .7,1.2)}) node[right] {$\bar x^3$};
    
        \draw[-] (.1,.8) -- (-.1,.8) node[left] {$s_0$};
    \draw[-] (-.1,-.8) -- (.1,-.8) node[right] {$-s_0$};
    
 \end{tikzpicture} 
 \caption{\footnotesize An illustration of the truncation procedure in the proof of Theorem \ref{thmNparticle} (iii): $\bar x^i$ are defined in \eqref{barxdef}. Dashed lines are the untruncated $x^i$.}\label{trunc}
\end{figure}

\begin{equation}\label{barxdef}
\bar x_t^i :=
\begin{cases}
R_0 + \frac{i}{N} & \text{if $x_t^i > R_0 + \frac{i}{N}$} \\
x_t^i & \text{if $-R_0 - \frac{N+1-i}{N} \le x_t^i \le R_0 + \frac{i}{N}$}\\
-R_0 - \frac{N+1-i}{N} & \text{if $x_t^i < -R_0 - \frac{N+1-i}{N}$}.
\end{cases}
\end{equation}
We first show that $\bar \bx$ belongs to $\cK_T^{S,N}$. Since  $-R_0 - \frac{N+1-i}{N} \le \bar x_t^{i} \le R_0 + \frac{i}{N}$,  there hold  
\[\begin{cases} \bar x_t^i \le x_t^i& \text{ if }x_t^i \ge  -R_0 - \frac{N+1-i}{N}\\ \bar x_t^i \ge x_t^i &\text{ if }x_t^i \le R_0 +\frac{i}{N}.\end{cases}\] By definition, we get 
\[
\bar x_t^{i+1} = 
\begin{cases}
R_0 + \frac{i+1}{N}  ( \ge \bar x_t^{i} + \frac1{N} )& \text{if $x_t^{i+1} > R_0 + \frac{i+1}{N}$}, \\
x_t^{i+1} (\ge x_t^{i} + \frac1{N}  \geq  \bar x_t^{i} + \frac1{N}) & \begin{array}{l}  \text{if } -R_0 - \frac{N-i}{N} \le x_t^{i+1} \le R_0 + \frac{i+1}{N} \\ \quad\text{and }  x_t^i \ge -R_0 - \frac{N+1-i}{N}, \end{array} \\
x_t^{i+1}( \ge -R_0 - \frac{N -i}{N}   = \bar x_t^{i} + \frac1{N}) & \begin{array}{l} \text{if }  -R_0 - \frac{N-i}{N} \le x_t^{i+1} \le R_0 + \frac{i+1}{N}  \\ \quad\text{and } x_t^i < -R_0 - \frac{N+1-i}{N}, \end{array} \\
-R_0 - \frac{N-i}{N} ( = \bar x_t^{i} + \frac1{N} )& \text{if $x_t^{i+1} < -R_0 - \frac{N-i}{N}$ \ (so $x_t^i < -R_0 - \frac{N+1-i}{N}$)},
\end{cases}
\]
and so we conclude that  $\bar x_t^{i+1} \ge \bar x_t^{i} + \frac1{N}$. 

Moreover, since $\bar x_t^{i} = \max\{\min\{x_t^i, R_0 + \frac{i}{N}\}, -R_0 - \frac{N+1-i}{N}\}$,
\begin{multline*}
\bar x_t^{N+1-i} = \max\left\{\min\left\{x_t^{N+1-i}, R_0 + \frac{N+1-i}{N}\right\}, -R_0 - \frac{i}{N}\right\} = \\-
\min\left\{\max\left\{x_{-t}^{i}, - R_0 - \frac{N+1-i}{N}\right\}, R_0 + \frac{i}{N}\right\} = - \bar x_{-t}^{i},
\end{multline*}
and $\bar x^i_{\frac T4 + t} = \bar x^i_{\frac T4 - t}$ follows analogously.

We can now evaluate $J_T^N(\bar \bx)$. Recalling that $W'(s) < 0$ for all $s < -R_0$ and $W'(s) > 0$ for all $s > R_0$ by symmetry assumptions on $W$ we get that \[
W(\bar x_t^i) =
\begin{cases}
W(R_0 + \frac{i}{N}) < W(x_t^i) & \text{if $x_t^i > R_0 + \frac{i}{N}$} \\
W(x_t^i) & \text{if $-R_0 - \frac{N+1-i}{N} \le x_t^i \le R_0 + \frac{i}{N}$}\\
W(-R_0 - \frac{N+1-i}{N}) < W(x_t^i) & \text{if $x_t^i < -R_0 - \frac{N+1-i}{N}$}.
\end{cases}
\]
Since we are supposing that $x^1_{\hat t}< -R_0 - 1 $, by continuity of $t \mapsto x^1_{\hat t}$ we have $x^1_{t}< -R_0 - 1$ for all $t$ in a neighborhood of $\hat t$, and therefore
\begin{equation}\label{decr1}
 \frac1N \sum_{i=1}^N \int_{-T/2}^{T/2} W(\bar x^i_t) dt <  \frac1N \sum_{i=1}^N \int_{-T/2}^{T/2}W(x^i_t) dt .
\end{equation}
In addition, by the definition \eqref{barxdef} of $\bar \bx$ it is clear that
\begin{equation}\label{decr2}
\frac1N \sum_{i=1}^N \int_{-T/2}^{T/2}|\dot{\bar x}^i_t|^2 dt \le \frac1N \sum_{i=1}^N \int_{-T/2}^{T/2} |\dot{x}^i_t|^2 dt.
\end{equation}
Finally, to evaluate the interaction energy, first of all we observe  that reciprocal distances between agents do not increase in this truncation procedure: indeed 
\[ 
\bar x_t^{i+1} - \bar x_t^i =
\begin{cases}
R_0 + \frac{i+1}{N} - R_0 -\frac{i}{N}  = \frac{1}{N} & \text{if $x_t^i > R_0 + \frac{i}{N}$ \ (so $x_t^{i+1} > R_0 + \frac{i+1}{N}$)}, \\
R_0 + \frac{i+1}{N}  - x_t^i \le x_t^{i+1}  - x_t^i  & \begin{array}{l}  \text{if $-R_0 - \frac{N+1-i}{N} \le x_t^i \le R_0 + \frac{i}{N}$}, \\ \quad\text{and }   x_t^{i+1}> R_0 +\frac{i+1}{N}, \end{array}  \\
 x_t^{i+1}  - x_t^i  & \begin{array}{l}  \text{if $-R_0 - \frac{N+1-i}{N} \le x_t^i \le R_0 + \frac{i}{N}$}, \\ \quad\text{and }   x_t^{i+1}\leq R_0 +\frac{i+1}{N}, \end{array}  \\
 x_t^{i+1} +R_0 + \frac{N+1-i}{N} < x_t^{i+1}  - x_t^i &\begin{array}{l}  \text{if } x_t^i < -R_0 - \frac{N+1-i}{N} \\ \quad\text{and }  x_t^{i+1} \ge -R_0 - \frac{N-i}{N}, \end{array} \\
-R_0 - \frac{N-i}{N} +R_0 + \frac{N+1-i}{N} = \frac{1}{N} &\begin{array}{l}  \text{if } x_t^i < -R_0 - \frac{N+1-i}{N} \\ \quad\text{and }  x_t^{i+1} < -R_0 - \frac{N-i}{N}, \end{array} 
\end{cases}
\]
so $\bar x_t^{i+1} - \bar x_t^i \le  x_t^{i+1}  - x_t^i $, that implies $|\bar x_t^{j} - \bar x_t^i| \le  |x_t^{j}  - x_t^i|$ for all $i, j, t$.
Therefore, monotonicity of $K(\cdot)$   yields  $K(|\bar x_t^{i}  - \bar x_t^j|) \ge K(|x_t^{i} - x_t^j|)$ for all $i,j,t$, and then \begin{equation}\label{decr3}
-\frac1{N^2}\sum_{i \neq j}  \int_{-T/2}^{T/2} K\big(|\bar x^i_t- \bar x^j_t|\big) dt \le - \frac1{N^2}\sum_{i \neq j}  \int_{-T/2}^{T/2}K\big(|x^i_t-x^j_t|\big) dt.
\end{equation}
Hence, adding \eqref{decr1}, \eqref{decr2} and \eqref{decr3} provides
\[
J_T^N(\bar \bx) < J_T^N(\bx),
\]
which contradicts the minimality of  $\bx$.
\end{proof}

 \subsection{Optimality conditions} \label{soptcond}
In this section we derive optimality conditions for the minimizers of the functional $J_T^N$, in the sets $\cK_T^N$ and $\cK_T^{N,S}$. These will be crucial in the proof of Theorem \ref{propositionsaturation}.
Due to the distance constraints,  the optimality conditions   are given  in terms of inequalities, and hold as equalities on intervals where the constraints is not saturated. 
First of all we start observing that,   as shown in the continuous case in Remark \ref{remreg}, in the optimality conditions no further  multipliers  due to the symmetry conditions appear. 
 
\begin{remark}\upshape \label{remregdiscrete}  Let $\bf x$ be a minimizer of $J_T^N$ on $\cK_T^{N,S}$. First of all we observe that $\bf x$ minimizes on $\cK_T^{N,S}$ the following linearized functional:
\begin{multline}\label{linj}
\overline J_T^N({\bf y};{\bf x}) := \frac1N \sum_{i=1}^N \int_{-T/2}^{T/2} \dot{x}^i_t \dot{y}^i_t  dt +  \frac1N \sum_{i=1}^N \int_{-T/2}^{T/2}  W'(x^i_t) y^i_t dt \\ - \frac1{N^2}\sum_{i \neq j} \int_{-T/2}^{T/2} K'\big(|x^i_t-x^j_t|\big)\frac{x^i_t-x^j_t}{|x^i_t-x^j_t|}(y^i_t-y^j_t) dt .
\end{multline}
To prove this, we fix ${\bf y} \in \cK_T^{N,S}$ and we denote by ${\bf y}^\lambda := (1-\lambda) {\bf x} + \lambda {\bf y}$ for all $\lambda\in [0,1]$. Since ${\bf y}^\lambda \in \cK_T^{N,S}$ there holds \begin{align*} 0\leq & \frac{J_T^N({\bf y}^\lambda) - J_T^N({\bf x})}{\lambda}=\frac1N \sum_{i=1}^N \int_{-T/2}^{T/2} \dot{x}^i_t   (\dot{y}^i_t - \dot{x}^i_t)dt +\frac{\lambda}{2N}\sum_{i=1}^N \int_{-T/2}^{T/2} |\dot{y}^i_t|^2-|\dot{x}^i_t |^2 dt \\ &+   \frac1N \sum_{i=1}^N \int_{-T/2}^{T/2} \frac{ W(x^i_t+\lambda(y^i_t-x^i_t))-W(x^i_t)}{\lambda} dt\\
& - \frac1{N^2}\sum_{i \neq j} \int_{-T/2}^{T/2} \frac{K\big(|x^i_t-x^j_t+\lambda (y^i_t-y^j_t-x^i_t+x^j_t) |\big)-K\big(|x^i_t-x^j_t  |\big) }{\lambda} dt  \end{align*} from which we conclude sending $\lambda \to 0$.

In  addition we get that ${\bf x}$ is a minimizer of $\overline J_T^N$ in a more general class of non-symmetric competitors:
\begin{equation}\label{optlin}
\overline J_T^N({\bf z};{\bf x}) \ge \overline J_T^N({\bf x}; {\bf x}) \qquad \text{for all ${\bf z} \in \cK_T^{N}$}.
\end{equation}
Indeed, let  ${\bf z} \in \cK_T^{N}$,  such that $z^i_{\frac T4 + t} = z^i_{\frac T4 - t}$. Define 
\[
y^i_t := \frac{z^i_t-z^{N+1-i}_{-t}}2, \qquad w^i_t := \frac{z^i_t +z^{N+1-i}_{-t}}2\qquad z^i_t = y^i_t + w^i_t. 
\]
It is straightforward to check that $y^i_t -y^{i-1}_t=\frac{z^i_t-z^{i-1}_{t}}2+\frac{z^{N+2-i}_{-t}-z^{N+1-i}_{-t}}2\geq \frac{1}{N}$, and $y^i_t=-y_{-t}^{N+1-i}$, so  ${\bf y} \in \cK_T^{N,S}$.

Moreover $w^{N+1-i}_t = w^i_{-t}$.  By this, and recalling 
$x^{N+1-i}_t =- x^i_{-t}$    and that $W'(s)=-W'(-s)$ by \eqref{ref}, we obtain    $\overline  J_T^N({\bf w};{\bf x})=0$. 
So,  by linearity, we conclude \[
\overline  J_T^N({\bf z}; {\bf x}) = \overline  J_T^N({\bf y}; {\bf x}) + \overline  J_T^N({\bf w};{\bf x}) = \overline  J_T^N({\bf y};{\bf x}) \ge \overline J_T^N({\bf x};{\bf x}) .
\]

To prove that the previous inequality holds without the assumption $z^i_{\frac T4 + t} = z^i_{\frac T4 - t}$, one can argue in a similar way, using the decomposition $z^i_t = \frac12(z^i_{t} + z^i_{\frac T 2 -t}) + \frac12(z^i_{ t} - z^i_{\frac T 2 -t })$.
\end{remark}

For $J=1, \ldots, N-1$, let us consider the mean of $(x_t^{J+1}, \ldots, x_t^N)$ and the mean of $(x_t^1, \ldots, x_t^{J})$, i.e.
\begin{equation}\label{bdef}
\ox_t^{J+1}:= \frac1{N-J} \sum_{i=J+1}^N x^i_t, \qquad \ux^J_t := \frac1{J} \sum_{i=1}^J x^i_t.
\end{equation}
\begin{figure}
\centering
\begin{tikzpicture}[scale=0.75]
      \draw[->] (-4,0) -- (4,0) node[above] {$t$};
      \draw[->] (0,-1.5) -- (0,1.5) node[left] {$\R$};
      
          \draw[scale=1,domain=-4:4,smooth,samples=50,variable=\x] plot ({\x},{atan(2*sin(500* \x /6.28))/70 - .65}); 
    \draw[scale=1,domain=-4:4,smooth,samples=50,variable=\x] plot ({\x},{atan(2*sin(500* \x /6.28))/70 - .4}); 
        \draw[scale=1,domain=-4:4,smooth,samples=50,variable=\x] plot ({\x},{atan(2*sin(500* \x /6.28))/70 }); 
    \draw[scale=1,domain=4:-4,smooth,samples=50,variable=\x] plot ({\x},{atan(2*sin(500* \x /6.28))/70 + .4}); 
        \draw[scale=1,domain=4:-4,smooth,samples=50,variable=\x] plot ({\x},{atan(2*sin(500* \x /6.28))/70 + .65}) ; 
        
         \draw	(-4, .4) node[anchor=east] {$ \ux^{J}_t \ \Big \{$};
         \draw	(-4, 1.3) node[anchor=east] {$ \ox^{J+1}_t \ \big \{$};

 \end{tikzpicture}
\caption{}\label{bdefpic}
\end{figure}
See Figure \ref{bdefpic}. We first show that if $\bx$ is a minimizer of $J_T^N$ either in $\cK_T^{N,S}$ or in $\cK_T^N$, then $\ox_t^{J+1}$ and $\ux_t^{J}$ solve useful differential inequalities.
\begin{lemma}\label{dmoptcont} Suppose that $\bf x$ minimizes $J_T^N$ on $\cK_T^{N,S}$ (or on $\cK_T^N$, in which case $(\bx_t)'\equiv 0$). Then, for any $J=1, \ldots, N-1$,
\begin{gather}\label{diffineq}
 \displaystyle -(\ox_t^{J+1})''+  \frac1{(N-J)} \sum_{i\ge J+1} W'(x^i_t) -\frac2{N(N-J)}\sum_{\substack{ i \ge J+1 \\ j \le J}} K'\big(|x^i_t-x^j_t|\big) \ge 0 \\
 \displaystyle (\ux_t^{J})''-  \frac1{J} \sum_{i\le J} W'(x^i_t) -\frac2{NJ}\sum_{\substack{ i \ge J+1 \\ j \le J}} K'\big(|x^i_t-x^j_t|\big) \ge 0
\end{gather}
in the weak sense in $\R$.

If, in addition, $x_t^{J+1} > x_t^{J} + \frac1{N}$ for some $J = 1, \ldots, N$ in an open set $\Omega \subseteq \R$, then 
\begin{gather*}
 \displaystyle -(\ox_t^{J+1})''+  \frac1{(N-J)} \sum_{i\ge J+1} W'(x^i_t) -\frac2{N(N-J)}\sum_{\substack{ i \ge J+1 \\ j \le J}} K'\big(|x^i_t-x^j_t|\big) = 0 \\
 \displaystyle (\ux_t^{J})''-  \frac1{J} \sum_{i\le J} W'(x^i_t) -\frac2{NJ}\sum_{\substack{ i \ge J+1 \\ j \le J}} K'\big(|x^i_t-x^j_t|\big) = 0
\end{gather*}
in the weak sense in $\Omega$.

\end{lemma}

\begin{proof}
If $\bf x$ is a minimizer of $J_T^N$ on $\cK_T^{N,S}$, by  Remark \ref{remregdiscrete}  $\bx$ is a minimizer of the linearized functional $\overline J_T^N({\bf y}; {\bf x})$ in $\cK_T^N$, so we do not impose  any symmetry conditions on  the competitors to derive the optimality conditions. 
Let $\psi \in C^\infty_{\rm per}(\R)$, $\psi \ge 0$ on $\R$, and
\[
\tilde x^i_t :=
\begin{cases}
x^i_t +  \psi(t)& \text{if $J+1 \le i \le N$}\\
x^i_t & \text{if $1 \le i \le J$}
\end{cases}
\qquad \text{for all $t$.}
\]
Note that $ \tilde x^{i}_t $ consists of trajectories where $x^i_t$ is shifted ``upwards'' whenever $i \ge J+1$. 
Since $\tilde x^{i+1} \ge \tilde x^{i} + \frac1{N}$, we have $\tilde \bx \in \cK_T^N$.

Note that
\begin{multline*}
 \sum_{i \neq j} \int_{-T/2}^{T/2} K'\big(|x^i_t-x^j_t|\big)\frac{x^i_t-x^j_t}{|x^i_t-x^j_t|}(\tilde x^ i_t- \tilde x^j_t) dt -  \sum_{i \neq j} \int_{-T/2}^{T/2} K'\big(|x^i_t-x^j_t|\big)\frac{x^i_t-x^j_t}{|x^i_t-x^j_t|}(x^ i_t-  x^j_t) dt  =  \\
- \int_{-T/2}^{T/2}  \sum_{\substack{ i \le J \\ j \ge J+1}}   K'\big(|x^i_t-x^j_t|\big)\frac{x^i_t-x^j_t}{|x^i_t-x^j_t|}\psi(t) + \int_{-T/2}^{T/2}  \sum_{\substack{ i \ge J+1 \\ j \le J}}  K'\big(|x^i_t-x^j_t|\big)\frac{x^i_t-x^j_t}{|x^i_t-x^j_t|}\psi(t) \\
= 2\int_{-T/2}^{T/2} \sum_{\substack{ i \ge J+1 \\ j \le J}} K'\big(|x^i_t-x^j_t|\big) \psi(t) dt.
\end{multline*}
Therefore, by the minimum property \eqref{optlin} of  $\bx$,
\begin{multline*}
0 \le \overline  J_T^N(\tilde \bx; {\bf x}) - \overline J_T^N(\bx; {\bf x}) = \\
 \int_{-T/2}^{T/2} \frac1N\sum_{i \ge J+1} \dot{x}^i_t \dot \psi(t) +  \frac1N \sum_{i\ge J+1} W'(x^i_t) \psi(t) dt -\frac2{N^2}\sum_{\substack{ i \ge J+1 \\ j \le J}} K'\big(|x^i_t-x^j_t|\big) \psi(t) dt .
\end{multline*}
Multiplying the previous inequality by $N/[N-J]$ yields the weak inequality for $\ox^{J+1}_t$.

To prove the weak inequality for $\ux^J_t$ one can argue similarly, considering competitors of the form
\[
\tilde x^{i}_t :=
\begin{cases}
x^i_t & \text{if $J+1 \le i \le N$}\\
x^i_t - \psi(t)& \text{if $1 \le i \le J$}
\end{cases}
\qquad \text{for all $t$.}
\]
Variations of $ \overline  J_T^N$ yield the following inequality in the weak sense
\[
\displaystyle (\ux_t^{J})''-  \frac1{J} \sum_{i\le J} W'(x^i_t) -\frac2{NJ}\sum_{\substack{ i \ge J+1 \\ j \le J}} K'\big(|x^i_t-x^j_t|\big) \ge 0 .
\]
in $\R$.


Finally, if $x_t^{J+1} > x_t^{J} + \frac1{N}$ for some $J = 1, \ldots, N$ in an open set $\Omega \subset (0,T/2)$, then it is possible to perturb trajectories in the opposite direction in space (i.e. ``downwards''), being the constraint $x_t^{J+1} = x_t^{J} + \frac1{N}$ inactive. In particular, for small $\eps > 0$ and $\psi \in C^\infty_{\rm per}(\R)$, it is possible to look at first order variations of the form
\[
\tilde x^{\eps, i}_t :=
\begin{cases}
x^i_t - \eps \psi(t)& \text{if $J+1 \le i \le N$}\\
x^i_t & \text{if $1 \le i \le J$}
\end{cases}
\qquad \text{for all $t \in \Omega$,}
\]
that yield the desired equalities.
\end{proof}

\begin{lemma}\label{mJconvex} Suppose that $\bf x$ minimizes $J_T^N$ on $\cK_T^N$  (or on $\cK_T^{N,S}$), and that $x_t^{J+1} > x_t^{J} + \frac1{N}$ for some $J = 1, \ldots, N$ in an open set $\Omega \subseteq \R$. Then, the map $t \mapsto \ox_t^{J+1} - \ux_t^{J}$ is of class $C^2(\Omega)$, and
\begin{multline}\label{eqMj}
 \displaystyle (\ux_t^{J} - \ox_t^{J+1})'' +  \frac1{(N-J)} \sum_{i\ge J+1} W'(x^i_t)-  \frac1{J} \sum_{i\le J} W'(x^i_t) \\ -\frac2{NJ}\sum_{\substack{ i \ge J+1 \\ j \le J}} K'\big(|x^i_t-x^j_t|\big)  -\frac2{N(N-J)}\sum_{\substack{ i \ge J+1 \\ j \le J}} K'\big(|x^i_t-x^j_t|\big) = 0
\end{multline}
in $\Omega$. Moreover,
\[
(\ox_t^{J+1} - \ux_t^{J} )'' \ge 2\min_{0 < r \le 2R_0 + 2} |K'(r)|  -2\max_{|x| \le R_0 + 1} |W'(x)| \qquad \text{in $\Omega$}.
\]
\end{lemma}

\begin{proof} The equality \eqref{eqMj} in the weak sense follows directly by Lemma \ref{dmoptcont}. Note that curves $x^i_t$ belong to $W^{1,2}(0,T)$, and therefore they are of class $C^{1/2}(\Omega)$. Thus, $t \mapsto \ox_t^{J+1} - \ux_t^{J}$ is of class $C^2(\Omega)$, with $1/2$-H\"older second order derivative, and \eqref{eqMj} holds in the classical pointwise sense.

Recall that in view of Theorem \ref{thmNparticle} (iii), optimal trajectories $x^i_t$ are bounded independently of $i, t, N$, that is $|x^i_t| \le R_0 + 1$ for all $t \in [0,T]$ and $i = 1, \ldots N$. So, in particular,  $\pm W'(x^i_t) \ge -\max_{|x| \le R_0 + 1} |W'(x)|$ for all $i, t$.
 On the other hand, $- K'\big(|x^i_t-x^j_t|\big) \ge - \max_{0 < r \le 2R_0 + 2} K'(r)=  \min_{0 < r \le 2R_0 + 2} |K'(r)|$, since $K$ is nonincreasing, Therefore 
 by \eqref{eqMj} we get  that in $\Omega$ there holds
\begin{align*}  (\ox_t^{J+1} - \ux_t^{J} )'' \ge& \ - \max_{|x| \le R_0 + 1} |W'(x)| \left( \frac1{(N-J)} \sum_{i\ge J+1}1+  \frac1{J} \sum_{i\le J} 1   \right)\\ & + \min_{0 < r \le 2R_0 + 2} |K'(r)|\left(
\frac2{NJ}\sum_{\substack{ i \ge J+1 \\ j \le J}}1+\frac2{N(N-J)}\sum_{\substack{ i \ge J+1 \\ j \le J}}1\right)\\
&= -2 \max_{|x| \le R_0 + 1} |W'(x)| +2  \min_{0 < r \le 2R_0 + 2} |K'(r)| .\end{align*} 

\end{proof}

\subsection{Saturation of the distance  constraint}
Finally, we show that minimizers saturate the distance constraints under the further assumption \eqref{coew2}.

\begin{proof}[Proof of Theorem \ref{propositionsaturation}] As for Theorem \ref{thmNparticle} (iii), we will prove the result only for $\bx\in \cK_T^{N,S} $ minimizing $J_T^N$ in $\cK_T^{N,S}$, being the other case is completely analogous (and even easier). Define, as in Appendix \ref{algapp},
\[
d^i_t = d^i(\bx_t) := x_t^{i} - x_t^{i-1}, \qquad \text{for $i=2, \ldots, N$},
\]
and
\[
\bm^J_t = \bm^J(\bx_t) :=  \ox_t^{J+1} - \ux_t^{J}=  \frac1{N-J} \sum_{i=J+1}^N x^i_t - \frac1{J} \sum_{i=1}^J x^i_t \qquad \text{for $J=1, \ldots, N-1$}.
\]
Let $\hat t$ be a time at which the (periodic) map $t \mapsto \bm^1_t$ reaches its global maximum on $\R$. We first claim that $d^2_{\hat t} = x^2_{\hat t} - x^1_{\hat t} = \frac1{N}$. If the claim is not true, then  $x^2_{t} - x^1_{t} > \frac1{N}$ for all $t$ in a neighborhood $\Omega$ of $\hat t$. 
In this case, by Lemma \ref{mJconvex},  $\bm^1_t$ is of class $C^2(\Omega)$, and since it reaches a maximum at $t= \hat t$, $(\bm^1)'_{\hat t} = 0$. Moreover by \eqref{coew2}
\[
(\bm^1_t)'' \ge 2 \min_{0 < r \le 2R_0 + 2} |K'(r)|  -2\max_{|x| \le R_0 + 1} |W'(x)|  > 0 \quad \text{in $\Omega$},
\]
but this contradicts the fact that $\bm^1_t$ has a maximum at $t= \hat t$. Then $d^2_{\hat t} = x^2_{\hat t} - x^1_{\hat t} = \frac1{N}$, and since $x^2_{t} - x^1_{t} \ge \frac1{N}$ for all $t$, the map $t \mapsto d^2_t$ has a global minimum at $t = \hat t$. In view of Lemma \ref{lemA1},
\[
\bm^2_t = \frac{N-1}{N-2}\bm_t^1 - \frac{N}{2(N-2)} d_t^{2},
\]
hence $\hat t$ is also a global maximum for $t \mapsto \bm^2_t$.

It is possible to show in the same way by induction that $\hat t$ is a global maximum for $t \mapsto \bm^i_t$ for any $i = 1, \ldots, N-1$, and that $d^j_{\hat t} =  \frac1{N}$ for all $j = 2, \ldots, N$. Indeed, suppose that the previous claim is true for $i , j \le J$; we need to show its validity for $i, j = J+1$. First, $d^{J+1}_{\hat t} = \frac1{N}$, otherwise $x^{J+1}_{t} - x^J_{t} > \frac1{N}$ for all $t$ in a neighborhood $\Omega$ of $\hat t$. Hence, by Lemma \ref{mJconvex},  $\bm^J_t$ is of class $C^2(\Omega)$, and $(\bm^J)'_{\hat t} = 0$. Moreover,
\[
(\bm^J_t)'' \ge 2\min_{0 < r \le 2R_0 + 2} |K'(r)|  -2\max_{|x| \le R_0 + 1} |W'(x)|   > 0 \quad \text{in $\Omega$},
\]
but this contradicts the fact that $\bm^J_t$ has a maximum at $t= \hat t$, yielding that $t \mapsto d^{J+1}_t$ has a global minimum at $t = \hat t$ that is equal to $\frac1{N}$. In view of Lemma \ref{lemA1},
\[
\bm^{J+1}_t = \frac{N-1}{N-J-1}\bm^1 - \sum_{j=1}^{J}\left(\frac{N-j}{N-J-1} - \frac j{J+1} \right)d^{j+1},
\]
and since $\left(\frac{N-j}{N-J-1} - \frac j{J+1} \right) > 0$, $\hat t$ is also a global maximum for $t \mapsto \bm^{J+1}_t$. 

Having obtained that $d^j_{\hat t} =  \frac1{N}$ for all $j = 2, \ldots, N$, we may conclude by Lemma \ref{lemA3} (with $\alpha = \frac1N$) that $\bm^J_{\hat t} =  \frac 1 {2}$ for all $J = 1, \ldots, N-1$. Since $\frac 1 {2}$ is the minimal value of $\bm^J_t$, and $\bm^J_{t} $ reaches its maximum at $\hat t$, the only possibility is $\bm^J_t = \frac 1 {2}$ for all $t \in [0, T]$.
Applying again  Lemma \ref{lemA3} we get that $x^{j}_t - x^{j-1}_t = d^j_{t} =  \frac1{N}$ for all $t$ and $j$.
\end{proof}

\begin{remark} \label{remenergia}\upshape We observe that, due to Theorem \ref{propositionsaturation},  the minimization problem \eqref{jtn} in $\cK_T^{N, S}$ is actually $1$-dimensional. 

We compute the energy of a minimizer  $\bx \in \cK_T^{N,S} $  of $J_T^N$ in $\cK_T^{N,S}$.  
We define the function
\begin{equation}\label{ewbarra} 
\overline W_N(x):=  \frac1N \sum_{i=0}^{N-1} W\left(x+\frac{i}{N}\right) \end{equation}
and the positive constant  \[ K^N:=\frac1{N^2}\sum_{i \neq j\in \{0,\dots, N-1\}}  K\left(\left|\frac{i-j}{N}\right|\right). \] 
Note that $K^N \le C_K$ for all $N$ by Proposition \ref{proplimite}.

Since by Theorem \ref{propositionsaturation} $x_t^i= x_t^1+\frac{i-1}{N}$ and $\dot x_t^i=\dot x_t^j$ for all $i, j,t$,  there holds \[
J_T^N(\bx) =  \int_{-T/2}^{T/2} \frac{|\dot{x}^1_t|^2 }{2} dt +  \int_{-T/2}^{T/2}\overline W_N \left(x^1_t\right) dt- K^N T.\] 
Moreover, again using that $x_t^i= x_t^1+\frac{i-1}{N}$, by the symmetry condition we get 
\[ x_t^1+\frac{i-1}{N}=x_t^i= -x_{-t}^{N+1-i}=-x_{-t}^{1}-\frac{N-i}{N}, \text{ \quad and then \ }x_t^1=-x_{-t}^1-\frac{N-1}{N}\ \forall t. \] This in particular forces 
\[x^1_0=-\frac{N-1}{2N}=x^1_{T/2}\quad \text{ and }\quad \dot x_t^1=\dot x_{-t}^1.\] 
Using these symmetries, and recalling that $W(s)=W(-s)$ we get
 \begin{equation}\label{jtnmin} J_T^N(\bx)= 2  \int_{0}^{T/2} \frac{|\dot{x}^1_t|^2 }{2} dt +  2\int_{0}^{T/2}\overline W_N \left(x^1_t\right) dt- K^N T:=\tilde J_T^1(x^1).\end{equation}
We observe that if $\bx$ is a minimizer of $J_T^N$ in $\cK_T^{N,S}$, then $x^1_t$ is a minimizer of $\tilde J^1_T$ on the set \[
\tilde \cK_T^{1,S} :=\left\{y\in W^{1,2}\left(\left[0, \frac{T}2\right]; \R\right)  \ :\  y_0= -\frac{N-1}{2N},\  y_{\frac T4 + t} = y_{\frac T4 - t} \ \forall \ t\in [0, T/4]\right\}.\]  Indeed if $y\in \tilde K_T^{1,S}$, we may construct  ${\bf y}\in W^{1,2}(\R; \R^N) $ as follows: first of all we extend $y$ on $[-T/2, 0]$ putting $y_{-t}=-y_t-\frac{N-1}{N}$, extend it to be a $T$ periodic function, and finally define $y_t^i:= y_t+\frac{i-1}{N}$ for $i=0, \dots , N$. It is easy to check that ${\bf y}\in \cK_T^{N,S}$ and that $J_T^N({\bf {y}})=\tilde J^1_T(y)$, This permits to conclude, by minimality of $\bx$.

Reasoning as in Remark \ref{remregdiscrete}, we get that $x^1_t$ is a minimizer of the following functional 
\[ \overline J_T^1(y; x^1) := 2\int_{0}^{T/2} \dot{x}^1_t \dot{y}_t  dt + 2  \int_{0}^{T/2}  \overline W_N'(x^1_t) y_t dt\]  among all $y\in   W^{1,2}((0, T/2); \R)$  such that  $y_0=y_{T/2}= -\frac{N-1}{2N}$. Indeed we   use the decomposition  $y_t = z_t+w_t:   = \frac12(y_{t} + y_{\frac T 2 -t}) + \frac12(y_{ t} - y_{\frac T 2 -t })$. It is easy to chech that $z_{t+T/4}=z_{T/4-t}$  and $z_0=z_{T/2}= \frac{N-1}{2N}$, so $w\in \tilde K_T^{1,S}$, whereas
$w_{t+T/4}=-w_{T/4-t}$ and $w_0=0= w_{T/2}$.  It is easy to check that  $\overline J_T^1(w)=0$, using the fact that $t\to x^1_{T/4+t}$ is even  whereas $t\mapsto w(T/4+t)$ is odd.  So, by linearity we get  $\overline J_T^1(y)= \overline J_T^1(z)$ and we conclude. 
 
Let  if $\phi\in C^{\infty}_c(0,T/2)$, then $y_t=x^1_t+\delta\phi(t)\in   W^{1,2}((0, T/2); \R)$ is such that  $y_0=y_{T/2}= -\frac{N-1}{2N}$. We compute  
\[0\leq \frac{\overline J_T^1(y; x^1)-\overline J_T^1(x^1;x^1)}{\delta} =2\int_{0}^{T/2} \dot{x}^1_t \dot{\phi}(t) dt + 2  \int_{0}^{T/2}  \overline W_N'(x^1_t) \phi(t) dt\]
and then $x_t^1$ is a  $T$-periodic solution to 
\begin{equation}\label{sys}\begin{cases} (x_t^1)''=  \overline W_N'(x^1_t)=\frac1N \sum_{i=0}^{N-1} W'\left(x^1_t+\frac{i}{N}\right)  & t\in\R 
\\ x_0^1= -\frac{N-1}{2N}=x_{T/2}^1.
\end{cases}
\end{equation}
Since $W\in C^1(\R)$ by \eqref{assw0},   $t\mapsto x_t^1$ is a $C^{2}$ function,  with $\sup_t | (x_t^1)''|\leq \max_{[-R_0-1, R_0+1]} |W'(s)|$. 
\end{remark}

\begin{remark}\label{nonsat}\upshape We now observe that non stationary minimizers may not saturate the distance constraint in general. Supposing that some given $K_0, W$ satisfy the standing assumption, then $K=\alpha K_0$ satisfies as well the same assumptions for any $\alpha > 0$. If $\alpha$ is large enough (and $K_0' < 0$), then Theorem \ref{propositionsaturation} states that minimizing trajectories minimize reciprocal distances. On the other hand, if $\alpha$ is small this property may not hold. Consider indeed a potential $W$ that, in addition to \eqref{assw0}, vanishes on $[-3,-2] \cup [2,3]$ and satisfies $W \ge 1$ on $[-1,1]$. Let $  \bx\in \cK_T^N$ be any minimizer of $J_T^N$ in $\cK_T^N$ and assume by contradiction that 
\begin{equation}\label{xdr}
x^{i+1}_t = x^{i}_t + \frac1{N}\qquad \text{for all $t$ and $i$} 
\end{equation}
holds. We first estimate $J_T^N(x)$ from above via a suitable competitor ${\bf y} \in \cK_T^N$. For $N$ odd,
\[
y^i_t :=
\begin{cases}
-2+\frac iN & i < \frac{N+1}2, \\
0 & i = \frac{N+1}2,\\
-\left(-2+\frac {N+1-i}N\right) & i > \frac{N+1}2.
\end{cases}
\]
Then,
\[
 \int_{-T/2}^{T/2} \frac{|\dot{x}^{\frac{N+1}2}_t|^2 }{2} dt +  \int_{-T/2}^{T/2}\frac1N \sum_{i} W(x^i_t) dt- \alpha K_0^N T\le J_T^N(x) \le J_T^N(y) \le \frac TNW(0).
\]
Note that $x^{\frac{N+1}2}_0 = 0$. Moreover, $\int_{-T/2}^{T/2} \frac{|\dot{x}^{\frac{N+1}2}_t|^2 }{2} dt$ can be bounded by a constant that does not depend on $N$ and $\alpha \le 1$, so there exists $\delta > 0$ such that $|x^{\frac{N+1}2}_t| \le 1/2$ for all $t \in (-\delta, \delta)$. By \eqref{xdr}, $|x^{i}_t| \le 1$ and $W(x^{i}_t) \ge 1$ for all $i$ and $t \in (-\delta, \delta)$, and $\delta$ does not depend on $N$. Thus,
\[
J_T^N(x) \ge \int_{-\delta}^{\delta}\frac1N  \sum_{i} W(x^i_t) dt- \alpha K_0^N T \ge 2\delta - \alpha C_{K_0} T,
\]
and for $\alpha$ small and $N$ large $J_T^N(x) \ge 2\delta - 2\alpha I(\chi_{(0,1)}) T >  \frac TNW(0) \ge J_T^N(y)$ which is a contradiction.

\end{remark}

\subsection{Some numerical experiments}\label{numan}

We now show a few numerical approximations of minimizers of the functional $J^N_T$ (on $\cK^{N,S}_T$). These have been obtained using a gradient descent method, with the following choice of the data:
\[
K_0(r) = \frac1{\sqrt{r}}, \qquad W(x) = 10\Big((0.5-x^2)^{a,+} + (x^2-3)^{a,+}\Big), \qquad T=50,
\]
where $(x)^{a,+} := (.1 \sqrt{1+(10x)^2}+ x)/2$ is a regularized version of the positive part of $x$. Though $W$ does not satisfy our standing assumptions (it does not have truly flat regions), it is very close to zero on $\pm[\sqrt{1/2}, \sqrt3$]. In our results, shown in Figure \ref{plots}, three different scenarios show up. When the non-local aggregative kernel is strong enough, that is when $K = 5 K_0$, periodic trajectories are evenly spaced on the whole time interval. This is in line with the assertion of Theorem \ref{propositionsaturation}. On the other hand, with a mild interaction kernel ($K = K_0$ in our pictures), agents split into two groups, that occupy the two wells (the example in Remark \ref{nonsat} is inspired by this scenario). When the number of agents  $N$ is odd, the middle $(N+1)/2$-th agent oscillates between the two groups which are located in the wells, while if $N$ is even then groups are simmetrically distributed. There is a final intermediate case, e.g. when $K = 3.1 K_0$, where the aggregation force is mild, and agents break down into two groups. In this case, it is not convenient for agents to saturate always the constraint, but to create some gap between two temporary groups while crossing the region of the potential between the two wells.

\begin{figure}
\centering
\begin{tabular}{cc}
\small $N = 18, \ K = 5 K_0$ & \small $N = 18, \ K = 3.1 K_0$ \\
\includegraphics[width=6.5cm]{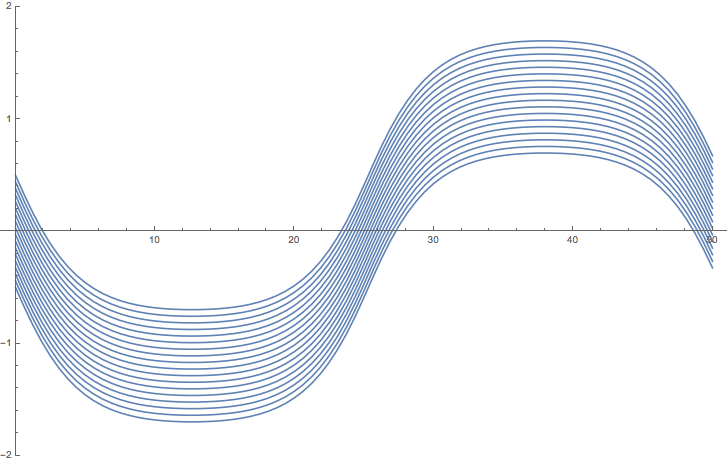} & \includegraphics[width=6.5cm]{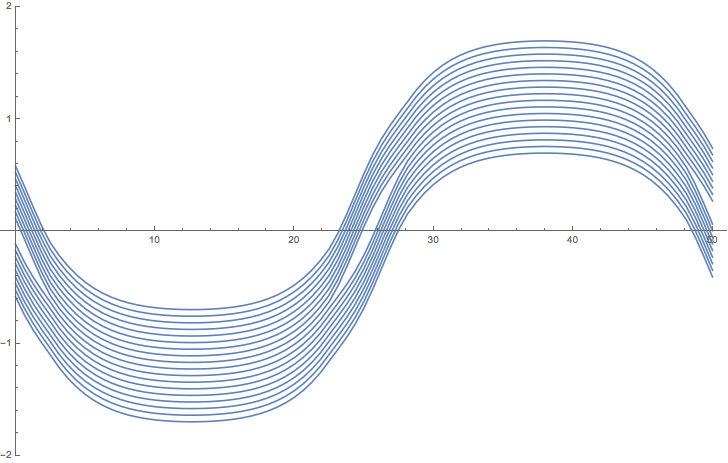} \\
\small $N = 18,\  K = K_0$ & \small $N = 17,\ K = K_0$ \\
\includegraphics[width=6.5cm]{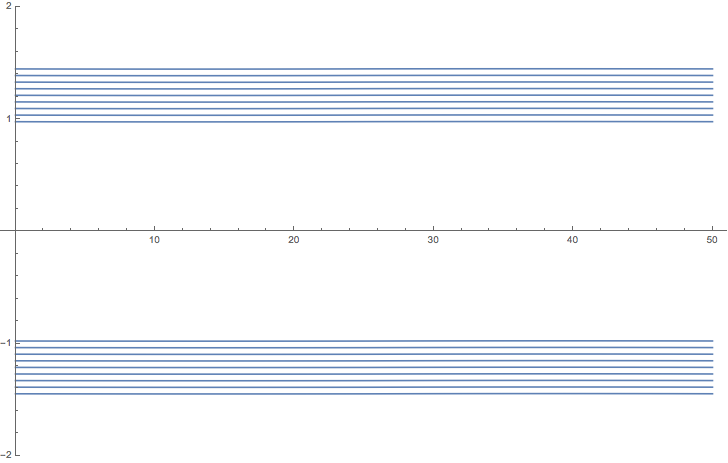} & \includegraphics[width=6.5cm]{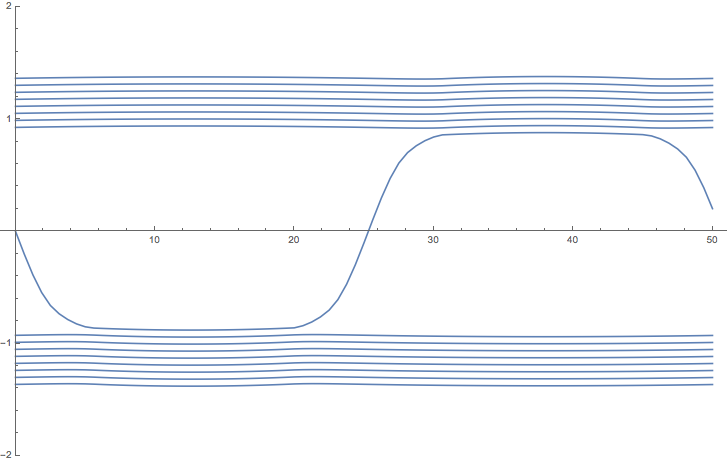} \\
\end{tabular}
\caption{\footnotesize Minimizers of $J_T^N$ in $\cK_T^{N,S}$ for different choices of the interaction kernel $K$ and number of agents $N$. Every line represents the evolution of a single agent. Time $t$ and position $x_t^i$ are in the horizontal and vertical axis respectively.}\label{plots}
\end{figure}

\section{$\Gamma$-convergence as $N\to +\infty$ and convergence of minimizers in dimension $1$} \label{sectiongamma}  
In this section we provide the   $\Gamma$-convergence  as $N\to +\infty$ of  the discrete energies  $J_T^{N} $  defined in \eqref{jtn} for $\bx \in \mathcal{K}_T^{N}$ (or  $\bx\in\cK_T^{N,S}$) to the energy  $J_T$, defined in \eqref{jt}, for $(m,w)\in \mathcal{K}_T$ (resp. $(m,w)\in \cK_T^S$). Moreover, we provide  
convergence of  minimal solutions of the $N$-agent system to minimal solutions of the mean-field problem. 
 This result is based on a standard application of  the $\Gamma$-convergence result,  and coercivity of the involved functionals. 
 
 First of all we extend the energies to include the constraints as follows 
\[J_T(m,w)=\begin{cases}J_T(m,w) & \text{if } (m,w)\in \mathcal{K}_T\\ +\infty &\text{ otherwise}   \end{cases}\qquad J_T^N(\bx)=\begin{cases} J_T^N(\bx) & \text{if } \bx\in \mathcal{K}_T^{N}\\ +\infty &\text{ otherwise}.   \end{cases}\] 
\begin{theorem}\label{gamma}  Let $T>0$. The following holds for the extended energies defined above. 
\begin{enumerate}
\item  Let $\bx = \bx^N \in \mathcal{K}_T^{N}$ be such that, as $N\to +\infty$,  $m_{\bx}^N\to \mu$  in $C([-T/2,T/2], \mathcal{P}(\R))$
and  $ \int_{-T/2}^{T/2}\delta_t\otimes w^N_\bx(t)dt\to \zeta \in \mathcal M([-T/2, T/2] \times \R)$ narrowly. Then    $\mu(t,dx)= m(t)dx$, with $m\in C(\R,  \cP_{2}^{r})$,  $T$-periodic $\zeta(dt,dx)= w(t,x)dt\otimes   dx$, $(m,w)\in \cK_T$ and  \[\liminf_N J_T^N(\bx)\geq J_T(m, w).\]
\item Let $(m,w)\in \mathcal{K}_T$. Then there exists  $\bx = \bx^N \in\cK_T^N $ such that as   $N\to +\infty$,  $m_{\bx }^N(\cdot)\to m(t,x)dx$  in $C([-T/2,T/2], \mathcal{P}_2(\R))$, $ \widetilde w^N_\bx = \int_{-T/2}^{T/2}\delta_t\otimes w^N_{\bx}(t)dt\to  w dt\otimes dx $ narrowly and \[\limsup_N J_T^N(\bx)\le  J_T(m,w).\]
\end{enumerate} 

In particular, \[J_T^N(\bx) \stackrel{\Gamma}{\longrightarrow}  J_T(m,w) \qquad \text{as $N\to +\infty$,} \] with respect  to  convergence in $C([-T/2,T/2], \mathcal{P}_2(\R))$ of the empirical measure $m_{\bx}^N(\cdot)$   to $m(\cdot)dx$  and  narrow  convergence of the  empirical measure $  \int_{-T/2}^{T/2}\delta_t\otimes w^N_\bx(t)dt$  to $w \  dt\otimes  dx$.

\end{theorem} 
For the proof of this result we will need some preliminary lemmata.  Let $\rho_\eps$ the following family of strictly positive mollifiers
\[
\rho_\eps(x) = (2\pi\eps)^{-1/2} e^{-x^2/(2\eps)} \qquad \text{on $\R$, $\eps > 0$}.
\]

\begin{lemma}\label{glslemma0} Let $(m,w)\in \mathcal{K}_T$ and $\eps > 0$. Then, $$m_\eps(t) := m(t) \star \rho_\eps \qquad w_\eps(t) := w(t) \star \rho_\eps  $$
satisfies
\begin{itemize}
\item[{\it i)}] $\int_{\R} \left| \frac{w_\eps(t,x)}{m_\eps(t,x)}\right|^2 m_\eps(t,x) \, dx \le \int_{\R} \left| \frac{w(t,x)}{m(t,x)}\right|^2 m(t,x) \, dx$ \, for all $t$,
\item[{\it ii)}] $(m_\eps, w_\eps) \in \mathcal{K}_T$,
\item[{\it iii)}] $\max_{[0,T] \times \R} m_\eps(t,x) < 1$.
\item[{\it iv)}] $m_\eps \to m$ in $C([0,T]; \cP_2(\R))$ and $w_\eps \to w$ in $L^2((0,T) \times \R)$ as $\eps \to 0$.
\end{itemize}

\end{lemma}

\begin{proof} Assertion {\it i)} is contained in \cite[Lemma 8.1.9]{ags}. Note that $m_\eps$ is $T$-periodic, and moreover the $ -\partial_t m_\eps + {\rm div}(w_\eps) = 0$ in the sense of distributions,  by linearity.  
Therefore, {\it ii)} holds. 

As for the upper bound on $m_\eps$, note first that $m_\eps(t,x) < 1$ for all $(t,x) \in \R\times \R$ since $\rho_\eps$ has full support on $\R$. In addition, $\rho_\eps$ is globally Lipschitz in the $x$-variable (with Lipschitz constant that may depend on $\eps$), and $\lim_{R \to \infty }\sup_{t} \int_{\R \setminus B(0,R)} m(t,x)dx = 0$, therefore the $L^1$ mass of $m(t)$ can be made arbitrarily small outside $B(0,R)$ uniformly in $t$ by choosing $R$ large. With these two facts one can control the sup-norm of $m$ in 
$\R \setminus B(0,R)$ by a small constant, see e.g. \cite[Lemma 2.2]{cc}. Namely, there is $R$ large such that $m(t,x) \le 1/2$ for all $(t,x) \in\R \times (\R \setminus B(0,R))$. On the other hand, since $m_\eps$ is continuous, $\max_{\R \times B(0, R)} m_\eps(t,x) < 1$, and {\it iii)} holds.

Finally, that $w_\eps \to w$ in $L^2(((-T/2, T/2) \times \R)$ as $\eps \to 0$ is standard. As for the convergence of $m_\eps \to m$ in $C(\R; \cP_2(\R))$, note that $m_\eps(\cdot)$ is uniformly continuous by \eqref{unmezzoh} and {\it i)}, so it is enough to check that $m_\eps(t) \to m(t)$ in $\cP_2(\R)$ for all $t$. To this aim, note that $m_\eps(t) \to m(t)$ narrowly, and $x^2 \star \rho_\eps  = x^2 + \eps^2 \int_\R y^2 \rho(y) dy$, therefore $\int_\R x^2 m_\eps(t,x) dx \to \int_\R x^2 m(t,x) dx$. We conclude by  the characterization of convergence in $\cP_2(\R)$ in Lemma \ref{equiconv}. 
\end{proof}

\begin{lemma}\label{glslemma1} Let $\eps > 0$, $(m,w)\in \mathcal{K}_T$, and $(m_\eps, w_\eps)$ as in Lemma \ref{glslemma0}. Let
\[
v^\eps(t, x) := \frac{w_\eps(t,x)}{m_\eps(t,x)} \qquad \text{on $\R \times \R$}.
\]
Then, for all $(s,x) \in [-T/2, T/2] \times \R$, the ODE
\[
X^\eps_s(x,s) = x, \quad \dot{X}^\eps_t(x,s) = v^\eps(t, X^\eps_t(x,s))
\]
has a unique maximal solution defined on $\R$. Finally, denoting by $X^\eps_t$ the flow $X^\eps_t(x) := X^\eps_t(x, 0)$,
\[
m_\eps(t) = (X^\eps_t)_\#m_\eps(0).
\]
\end{lemma}

\begin{proof} All the assertions are proven in \cite[Proposition 8.1.8]{ags} in the a.e. sense, namely the ODE is shown to have a maximal solution for $m_\eps(s)$-a.e. initial datum $x$; since $m_\eps(s)$ is continuous and $\inf_{\R\times [-R,R]}m_\eps(t,x)>0$ for all $R>0$, this means that the ODE system has a unique global solution for a.e. $x$. Suppose now that $t \mapsto |X_t(x)|$ is unbounded for some $x$ as $t \to \bar t$. Since if $y < x < z$ we have $X_t(y) < X_t(x) < X_t(z)$ for all $t < \bar t$, $\liminf_{t \to \bar t} X_t(y) = -\infty$ for all $y < x$ or $\limsup_{t \to \bar t} X_t(z) = +\infty$ for all $z > x$, contradicting the fact that $X_t(\cdot)$ must be a.e. globally defined. Hence the ODE system has a unique global solution for all initial data.

\end{proof}

\begin{lemma}\label{glslemma2} Suppose that $m \in \cP \cap L^\infty(\R)$ satisfies $m > 0$ a.e. on $\R$. Let $R > 0$ and $\bx \in \R^N$ be such that
\[
\int_{-R}^{x^i} m(x) dx = \frac{M_R}{2N} + (i-1)\frac{M_R}N \qquad \text{for all $i = 1, \ldots, N$},
\]
where
\[
M_R := \int_{-R}^{R} m(x) dx \le 1.
\]
Then, $M_R m^N_\bx \to m\chi_{(-R,R)}$  narrowly as $N \to \infty$.
\end{lemma}

\begin{proof}
First of all,   $\bx$ is uniquely determined by continuity and strict monotonicity of the function $x \mapsto \int_{-R}^{x} m$ on $[-R,R]$, whose range is $[0, M_R]$. Fix any $\varphi \in C_b(\R) \cap {\rm Lip}(\R)$, and denote by $L > 0$ its Lipschitz constant. 

Note that, since $\int_{x^{i-1}}^{x^i} m = \frac{M_R}{N}$,
\[
\int_{(x^{i-1}, {x^i}]} \varphi(x)(M_R m^N_\bx - m)(dx) =  \int_{x^{i-1}}^{x^i}[\varphi(x^i) - \varphi(x)]m(x)dx,
\]
hence
\[
\left| \int_{(x^{i-1}, {x^i}]} \varphi(x)(M_R m^N_\bx - m)(dx)  \right| \le(x^i-x^{i-1})\frac {LM_R}{N}.
\]
Then, since $\int_{-R}^{x^1} m+\int_{x^N}^Rm = \frac{M_R}{N}$,
\begin{multline*}
\left|\int_{(-R,R)} \varphi(x)(M_Rm^N_\bx - m)(dx)\right| \le \left| \int_{(-R, x^1]} \varphi(x)(M_Rm^N_\bx - m)(dx)\right| + \\ + \sum_{i=2}^N \left| \int_{(x^{i-1}, {x^i}]} \varphi(x)(M_Rm^N_\bx - m)(dx)  \right| + \left| \int_{(x^N, R)} \varphi(x)(M_R m^N_\bx - m)(dx) \right| \le \\
\frac{M_R}{N}|\varphi(x^1)| + \int_{-R}^{x^1} |\varphi| m (dx) + \sum_{i=2}^N \left( (x^i-x^{i-1})\frac {LM_R}{N} \right) + \int_{x^N}^{R} |\varphi| m (dx) \\
\le 2\frac{\|\varphi\|_\infty M_R}{N} + (x^N - x^1)\frac {LM_R}{N} \le  \frac{2M_R}{N} (\|\varphi\|_\infty +RL),
\end{multline*}
thus $\int_{\R} \varphi(x)(M_Rm^N_\bx - m\chi_{(-R,R)})(dx) \to 0$. Since Lipschitz functions are dense in the space of continuous functions on $(-R,R)$, a standard approximation argument yields the same convergence for all $\varphi \in C_b(\R)$, that is the desired claim on narrow convergence of $M_R m^N_\bx$ to $m\chi_{(-R,R)}$.
\end{proof}

 We are now ready to prove the $\Gamma$-convergence result.

\begin{proof}[Proof of Theorem \ref{gamma}]\ \ \\

$\bullet$ {\it $\Gamma-\liminf$ inequality:  Given  $\bx \in \mathcal{K}_T^{N}$ such that as $N\to +\infty$,  $m_{\bx}^N\to \mu$ narrowly  
and that  $ \int_{-T/2}^{T/2}\delta_t\otimes w^N_\bx(t)dt\to \zeta \in \mathcal M([-T/2, T/2] \times \R)$ narrowly,  we have  to prove that  $\mu(t,dx)= m(t)dx$, with $m\in C(\R,  \cP_{2}^{r})$,  $\zeta(dt,dx)= w(t,x)dt\otimes   dx$, and  $\liminf_N J_T^N(\bx)\geq J_T(m, w)$. }

Assume  that $\liminf_N J_T^N(\bx) < \infty$, otherwise the statement is trivial. 
We recall that 
\[
\frac{1}{N}\sum_{i=1}^N\int_{-T/2}^{T/2} |\dot{x}^i_t|^2 dt =  \int_{-T/2}^{T/2}\int_{\R}  \left|\frac{dw^N_\bx(t)}{dm_{\bx}^N(t)}\right|^2 m_{\bx}^N(t)(dx)dt,
 \] and moreover by Proposition \ref{proplimite} $  \int_{-T/2}^{T/2} \mathcal{I}(m^N_\bx(t))dt\leq C_K T$. So,  by Lemma \ref{lemmalsc}, we get 
\begin{align}\nonumber
\infty > \liminf_{N \to \infty}\left( J_T^N(\bx) - \int_{-T/2}^{T/2} \int_{\R^d} W(x)m_\bx^N(t)dxdt+\int_{-T/2}^{T/2} \mathcal{I}(m_\bx^T(t)) \, dt\right) \\ = \liminf_{N \to \infty} \frac{1}{2N}\sum_{i=1}^N\int_{-T/2}^{T/2} |\dot{x}^i_t|^2 dt=\liminf_N\int_0^T\int_{\R}  \left|\frac{dw^N_\bx(t)}{dm_{\bx}^N(t)}\right|^2 m_{\bx}^N(t)(dx)dt\geq \int_{[-T/2, T/2]\times\R}\left|\frac{d \zeta}{d \mu}\right|^2 d \mu. \label{limi0}
\end{align}
As a consequence, $\zeta$ is absolutely continuous with respect to $\mu$. Moreover, by Proposition \ref{convparticelle}, for all $t \in\R$, $\mu(t,dx) = m(t) dx$ for some $m(t) \in L^\infty(\R)$, and $\sup_t \|m(t)\|_{L^\infty(\R)} \le 1$. Hence  $\zeta = w dt \otimes  dx$ for some $\mu$-measurable function $w$. Thus
\[
\infty >\int_{[-T/2, T/2]\times\R}\left|\frac{d \zeta}{d \mu}\right|^2 d \mu = \int_{-T/2}^{T/2} \int_{\R^d} \frac{|w(t,x)|^2}{m(t,x)} dxdt \ge \int_{-T/2}^{T/2}  \int_{\R^d} |w(t,x)|^2 dxdt,
\]
namely $w \in L^2((-T/2, T/2) \times \R)$. Finally, it is straightforward to check that $-\partial_t m^N_{\bx} + {\rm div}(w^N_{\bx}) = 0$ in the sense of distributions, and the equation passes to the limit by narrow convergence, and that $T$-periodicity of  $m$ is directly inherited by analogous property of $\bx_t$. Thus,  $(m, w) \in \cK_T$.

By Proposition \ref{proplimite} and the Dominated Convergence Theorem, we get that \begin{equation}\label{limi1} \lim_{N \to \infty} \frac1{N^2}\int_{-T/2}^{T/2} \sum_{i \neq j}  K\big(|x^i_t-x^j_t|\big) dt=\int_{-T/2}^{T/2}  \mathcal{I}(m(t))dt,\end{equation} since  
$\frac1{N^2}\sum_{i \neq j}  K\big(|x^i_t-x^j_t|\big)=\mathcal{I}(m_{\bx}^N(t))\to \mathcal{I}(m(t))$ for all $t\in[0, T]$ and $$\sup_{N, t}  \frac1{N^2}\sum_{i \neq j}  K\big(|x^i_t-x^j_t|\big) = \sup_{N, t}  \cI(m^N_{\bx}(t)) \leq C_K.$$  
Moreover,  by Lemma \ref{lscW} and Fatou's lemma, we get that \begin{equation}\label{limi2}\liminf_{N \to \infty}\int_{-T/2}^{T/2} \int_{\R^d} W(x)m_{\bx}^N(t)(dx)dt\geq\int_{-T/2}^{T/2} \int_{\R^d} W(x)m(t,x)dxdt.\end{equation} 
So, by \eqref{limi0}, \eqref{limi1} and  \eqref{limi2}  we conclude 
\[\infty > \liminf_N J_T^N(\bx) \geq  J_T(m,w).\]
 
 \medskip

$\bullet$ {\it  $\Gamma-\limsup$ inequality.

Given $(m,w)\in \mathcal{K}_T$, we have to prove that there exists $\bx_t\in\cK_T^N $ such that as   $N\to +\infty$,  $m_{\bx }^N(\cdot)\to m(t,x)dx$  in $C_{\text{per}}(\R, \mathcal{P}_2(\R))$, $ \widetilde w^N_\bx = \int_0^T\delta_t\otimes w^N_{\bx}(t)dt\to  w dt\otimes dx $ narrowly and $\limsup_N J_T^N(\bx)\le  J_T(m,w)$. }

We divide the proof into three steps. 

{\bf Step 1: regularization and localization.} For all $\eps > 0$, consider the regularized couple $(m_\eps, w_\eps) \in \cK_T$ as in Lemma \ref{glslemma0}. Letting $v_\eps = w_\eps/m_\eps$, consider the ODE flow $X_t$ defined in Lemma \ref{glslemma1}, induced by the velocity field $v_\eps$. Since $m_\eps(t) = (X^\eps_t)_\# m_\eps(0)$ for all $t$,
\begin{multline*}
\int_{-T/2}^{T/2}  \int_{\R} \left|\frac{w_\eps(t,x)}{m_\eps(t,x)}\right|^2 m_\eps(t,x) \, dx dt = \int_{-T/2}^{T/2}  \int_{\R} \left|\frac{w_\eps(t,x)}{m_\eps(t,x)}\right|^2  (X^\eps_t)_\# m_\eps(0) (dx) \, dt = \\ \int_{-T/2}^{T/2} \int_{\R} |v_\eps(t, X_t^\eps(y))|^2 m_\eps(0) (dy) \, dt =  \int_{\R} G_\eps(y) m_\eps(0, y) dy,
\end{multline*}
where
\begin{equation}\label{Gdef}
G_\eps(y) := \int_{\R} |v_\eps(t, X^\eps_t(y))|^2dt.
\end{equation}
Note that $G_\eps(y)$ is a non-negative continuous function on $\R$. Therefore, by the monotone convergence theorem,
\begin{equation}\label{Rconv}
\lim_{R \to \infty} \int_{-R}^R G_\eps(y) m_\eps(0, y) dy= \int_{-T/2}^{T/2}  \int_{\R} \left|\frac{w_\eps(t,x)}{m_\eps(t,x)}\right|^2  m_\eps(t,x) \, dx dt .
\end{equation}

{\bf Step 2: construction of the $N$-agents system.} Let $\eps > 0$ be fixed, and  $\eta_\eps=\max_{\R\times \R} m_\eps<1$, in view of   Lemma \ref{glslemma0} (iii). 
Then there exists $R_\eps$ such that  for every $R>R_\eps$ there holds \[M_{R_\eps, \eps} := \int_{-R}^{R} m_\eps(0, x) dx>\eta_\eps.
\]
For any $N \in \mathbb N$, define $\bx_0^{N, R, \eps}$ as follows:
\[
\int_{-R}^{(x_0^{N, R, \eps})^i} m_\eps(0, x) dx = \frac{M_{R, \eps}}{2N} + (i-1)\frac{M_{R, \eps}}N \qquad \text{for all $i = 1, \ldots, N$}.
\]
Let us then define the $N$-agents competitor as the system of trajectories following the velocity field $v_\eps$ and having initial datum $\bx_0^{N, R, \eps}$ at $t = 0$, i.e.
\[
\bx_t^{N, R, \eps} := X^\eps_t (\bx_0^{N, R, \eps}).
\]
Note that  $X^\eps_t\Big(\left[(x_0^{N, R, \eps})^{i}, (x_0^{N, R, \eps})^{i + 1}\right]\Big) =\left [(x_t^{N, R, \eps})^{i}, (x_t^{N, R, \eps})^{i + 1}\right]$, for all $i = 1, \ldots, N-1$. So we get   for all $t$
\begin{multline}\label{mutdist}
(x_t^{N, R, \eps})^{i+1}- (x_t^{N, R, \eps})^i \ge \frac1{\eta_\eps} \int_{(x_t^{N, R, \eps})^i}^{(x_t^{N, R, \eps})^{i+1}} m_\eps(t, x) dx = \frac1{\eta_\eps} \int_{(x_t^{N, R, \eps})^i}^{(x_t^{N, R, \eps})^{i+1}} (X^\eps_t)_\# m_\eps(0) (dx) \\ = \frac1{\eta_\eps} \int_{(x_0^{N, R, \eps})^i}^{(x_0^{N, R, \eps})^{i+1}} m_\eps(0,y) (dy) 
= \frac{M_{R, \eps}}{\eta_\eps N}>\frac{1}{N}
\end{multline}
since $R>R_\eps$. 
Moreover, using the fact that $m_\eps$ is $T$-periodic, we get $T$-periodicity of $\bx^{N,R,\eps}$. Indeed, denoting for simplicity $\bx^{N,R, \eps}$ as $\bx$ and recalling that $m_\eps(t) = (X^\eps_t)_\# m_\eps(0)$ and that $X^\eps_t((-\infty, y)) = (-\infty, X^\eps_t(y))$ for all $y$,  we get
\[\int_{-\infty}^{x_t^i} m_\eps(t,x)dx=\int_{-\infty}^{x_{t+T}^i} m_\eps(t+T, x) dx =\int_{-\infty}^{x_{t+T}^i} m_\eps(t, x) dx \]
which by positivity of $m_\eps$ implies that $x^i_t=x^i_{t+T}$ for all $i$ and $t$. 

 So,  $\bx^{N, R, \eps} \in \cK^N_T$ whenever $R>R_\eps$.

{\bf Step 3: convergence.} First, for $  \eps, R$ fixed, consider the flow of empirical measures $m^{N,R,\eps} := m_{\bx^{N,R,\eps} }^N$ associated to $\bx^{N,R,\eps}$. It holds $m^{N,R,\eps}(t) = (X_t^\eps)_\# m^{N,R,\eps}(0)$. Since $\dot{\bx}_t^{N,R,\eps} = v_\eps(t, \bx_t^{N,R,\eps})$,
\[
\frac1N \sum_{i=1}^N \int_{-T/2}^{T/2}  |(\dot{x}_t^{N,R,\eps})^i|^2 dt = \int_{-T/2}^{T/2}  |v_\eps(t,x)|^2m^{N,R,\eps}(t)(dx) \, dt
= \int_{\R} G_\eps(y) m^{N,R,\eps}(0) (dy),
\]
where $G$ is defined in \eqref{Gdef}. By Lemma \ref{glslemma2}, $m^{N,R,\eps}(0) \to M_{R,\eps}^{-1} m_\eps(0) \chi_{(-R,R)}$  narrowly as $N \to \infty$, therefore
\[
\lim_{N \to \infty} \frac1N \sum_{i=1}^N \int_{-T/2}^{T/2}  |(\dot{x}_t^{N,R,\eps})^i|^2 dt = \frac1{M_{R,\eps}} \int_{-R}^R G_\eps(y) m_\eps(0, y) dy.
\]
Recalling \eqref{Rconv}, Lemma \ref{glslemma0} (ii), and the fact that $M_{R,\eps} \to 1$ as $R \to \infty$,
\begin{multline*}
\lim_{R \to \infty} \lim_{N \to \infty} \frac1N \sum_{i=1}^N \int_0^T |(\dot{x}_t^{N,R,\eps})^i|^2 dt  =\int_{-T/2}^{T/2}  \int_{\R} \left|\frac{w_\eps(t,x)}{m_\eps(t,x)}\right|^2m_\eps(t,x) \, dx dt \\ \le \int_{-T/2}^{T/2}  \int_{\R} \left|\frac{w(t,x)}{m(t,x)}\right|^2m(t,x) \, dx dt.
\end{multline*}

Next, for $\eps,R > 0$ fixed, $m^{N,R,\eps}(t)$ converges narrowly to $(X_t^\eps)_\# \big( M_{R,\eps}^{-1} m_\eps(0) \chi_{(-R,R)} \big)$ as $N \to \infty$ for all $t$, and such a convergence is actually in $\cP_2(\R)$ since supports are compact. Note that the sequence $\{m^{N,R,\eps}(\cdot)\}_N$ is equicontinuous by Remark \ref{remreg}, since by Lemma \ref{glslemma0}  there holds \begin{multline*}
d_2^2(m^{N,R,\eps}(t), m^{N,R,\eps}(s)) \le |t-s|\int_s^t \int_\R |v_\eps(\tau,y)|^2 m_\eps(\tau, y) d y d\tau \\ \le  |t-s| \int_0^T \int_{\R} \left|\frac{w(t,x)}{m(t,x)}\right|^2m(t,x) \, dx dt.
\end{multline*}
Therefore, we get that 
\begin{equation}\label{mNconv}
m^{N,R,\eps} \to (X_\cdot^\eps)_\# \big( M_{R,\eps}^{-1} m_\eps(0) \chi_{(-R,R)} \big) \qquad \text{in \quad $C(\R; \cP_2(\R))$ as $N \to \infty$}. 
\end{equation}
Similarly, for $\eps > 0$ fixed, $(X_t^\eps)_\# \big( M_{R,\eps}^{-1} m_\eps(0) \chi_{(-R,R)} \big) = M_{R,\eps}^{-1} m_\eps(t,x) \chi_{(X^\eps_t(-R), X^\eps_t(R))}(x) dx$ converges to $m_\eps(t,x) dx$ in $C(\R; \cP_2(\R))$ as $R \to \infty$. Indeed, convergence in $\cP_2(\R)$ for fixed $t$ holds by narrow convergence and convergence of second order moments, which can be easily verified, and equicontinuity in $t$ can be obtained as before, recalling Remark \ref{remreg}:
\begin{multline*}
d_2^2(M_{R,\eps}^{-1} m_\eps(t,x) \chi_{(X^\eps_t(-R), X^\eps_t(R))}(x) dx , M_{R,\eps}^{-1} m_\eps(s,x) \chi_{(X^\eps_s(-R), X^\eps_s(R))}(x) dx ) \le\\
\frac{ |t-s|}{M_{R,\eps}}\int_s^t \int_{X^\eps_\tau(-R)}^{X^\eps_\tau(R)}  |v_\eps(\tau,y)|^2 m_\eps(\tau, y) d y d\tau \le 2 |t-s| \int_{-T/2}^{T/2}  \int_\R \left|\frac{w(t,x)}{m(t,x)}\right|^2m(t,x) \, dx dt.
\end{multline*}
Furthermore, by Lemma \ref{glslemma0} (iv), $m_\eps(t,x) dx$ converges to $m(t,x) dx$ in $C(\R; \cP_2(\R))$ as $\eps \to 0$.

We now pass to the convergence of $ \widetilde w^{N,R,\eps}_\bx =\int_{-T/2}^{T/2} \delta_t\otimes w^N_{\bx^{N,R,\eps}}(t)dt$. First, for $\eps, R >0 $ fixed, note that $\widetilde w^{N,R,\eps}_\bx = v_\eps m^{N,R,\eps}$; in particular, $w^{N,R,\eps}_\bx$ is compactly supported on $[-T/2,T/2] \times \R$, uniformly in $N$. Hence, by \eqref{mNconv}, $w^{N,R,\eps}_\bx$ converges narrowly to $M_{R,\eps}^{-1} v_\eps m_\eps(t,x) \chi_{(X^\eps_t(-R), X^\eps_t(R))}(x) dx$. For fixed $\eps > 0$, such a measure converges narrowly to $v_\eps m_\eps dt \otimes dx$ as $R \to \infty$ by the Dominated Convergence Theorem. Finally, $w_\eps = v_\eps m_\eps$ converges to $w$ in $L^2((-T/2,T/2) \times \R)$.

It is then possible to choose sequences $\eps_N \to 0$, $R_N \to \infty$ such that $\bx^{N,R_N,\eps_N} \in \cK^N_T$, $m^{N,R_N,\eps_N}$ converges to $m(t,x) dx$ in $C(\R; \cP_2(\R))$, $T$-periodic, and $\widetilde w^{N,R_N,\eps_N}$ converges to $w  dt \otimes  dx$ as $N \to \infty$, and
\begin{equation}\label{kinlimsup}
\limsup_{N \to \infty} \frac1N \sum_{i=1}^N \int_{-T/2}^{T/2}  |(\dot{x}_t^{N,R_N,\eps_N})^i|^2 dt  \le \int_{-T/2}^{T/2} \int_{\R} \left|\frac{w(t,x)}{m(t,x)}\right|^2m(t,x) \, dx dt.
\end{equation}
 
Finally, convergence of $m^{N,R_N,\eps_N}$ to $m(t,x) dx$ in $C(\R; \cP_2(\R))$ guarantees convergence of $\iint W m^{N,R_N,\eps_N}$ to $ \iint W m $ and of $\int \cI(m^{N,R_N,\eps_N})$ to $\int \cI(m)$ in view of Lemma \ref{lscW}, Proposition \ref{proplimite} (ii) and the Dominated Convergence Theorem. This yields, together with \eqref{kinlimsup},
\begin{multline*}
\limsup_N J_T^N(\bx) = \limsup_N  \left(\frac1N \sum_{i=1}^N \int_{-T/2}^{T/2}  \frac{|(\dot{x}_t^{N,R_N,\eps_N})^i|^2}{2} dt  +\int_{-T/2}^{T/2}   W(x) m^N_\bx(t)) dt  -\int_{-T/2}^{T/2} \mathcal{I}(m_{\bx}^N(t)) dt  \right) \\
\le  \limsup_N  \left(\frac1{2N} \sum_{i=1}^N \int_{-T/2}^{T/2} |(\dot{x}_t^{N,R_N,\eps_N})^i|^2 dt \right) + \lim_N\left( \int_0^T  W(x) m^N_\bx(t)) dt  \right) -\lim_N\int_{-T/2}^{T/2} \mathcal{I}(m_{\bx}^N(t)) dt\\
\le \int_{-T/2}^{T/2}  \int_{\R} \frac12\left|\frac{w(t,x)}{m(t,x)}\right|^2m(t,x) \, dx dt +\int_{-T/2}^{T/2} \int_{\R}  W(x) m(t,x)dxdt  -\int_{-T/2}^{T/2} \mathcal{I}(m(t))dt =  J_T(m,w).
\end{multline*}

\end{proof}
\begin{remark}\upshape \label{remarksim} We observe that the same  $\Gamma$-convergence result of Theorem \ref{gamma} can be obtained by restricting the energies to the constraints $\cK_T^S$ and $\cK_T^{S,N}$. 

Indeed in the $\Gamma$-liminf inequality, we get that if $\bx\in \cK_T^{S,N}$ then $(m,w)\in \cK_T^S$.  Note that  symmetry of $m$ with respect to $T/4$ is directly inherited by analogous properties of $\bx_t$. In addition, for any test function $\varphi$,
\begin{multline*}
\int_\R \varphi(x) m^N_{\bx}(-t) (dx) = \frac1N \sum_{i=1}^{N} \varphi(x^i_{-t}) = \frac1N \sum_{i=1}^{N} \varphi(-x^{N+1-i}_{t}) \\ = \frac1N \sum_{j=1}^{N} \varphi(-x^{j}_{t}) = \int_\R \varphi(-x) m^N_{\bx}(t) (dx),
\end{multline*}
and therefore in the limit $m(-t) = \gamma m(t)$.

As for the $\Gamma$-limsup, it is easy to check that if $(m,w)\in \cK_T^S$, then the $N$-agents system  $\bx_t^{N,R,\eps}$ constructed in Step 2 satisfies the symmetry assumptions, and then $\bx_t^{N, R, \eps}\in \cK_T^{S,N}$.
Indeed  denoting for simplicity $\bx_t = \bx_t^{N, R, \eps}$, by the definition of $\bx_0$ and $m_\eps(0,x) = m_\eps(0,-x)$, we have that 
\begin{multline*}
\int_{-R}^{x_0^{N+1-i}} m_\eps(0, x) dx = \frac{M_{R, \eps}}{2N} + (N-i)\frac{M_{R, \eps}}N = M_{R, \eps} - \int_{-R}^{x_0^i} m_\eps(0, x) dx \\ = \int_{x_0^i}^R m_\eps(0, x) dx = \int^{-x_0^i}_{-R} m_\eps(0, x) dx,
\end{multline*}
therefore $x_0^{N+1-i} = -x_0^i$ by the positivity of $m_\eps(0)$ on $\R$. Hence, since $m_\eps(t,x) = m_\eps(-t,-x)$ and $X^\eps_t((-\infty, y)) = (-\infty, X^\eps_t(y))$ for all $y$ (recall also the formula $m_\eps(t) = (X^\eps_t)_\# m_\eps(0)$),
\begin{multline*}
\int_{-\infty}^{x_t^{N+1-i}} m_\eps(t, x) dx = \int_{-\infty}^{x_0^{N+1-i}} m_\eps(0, x) dx =\int_{-\infty}^{-x_0^{i}} m_\eps(0, x) dx   
= \int_{x_0^i}^\infty m_\eps(0, x) dx \\ = \int_{x_t^i}^\infty m_\eps(t, x)dx
 = \int^{-x_t^{i}}_{-\infty} m_\eps(t, -x) dx =  \int^{-x_{-t}^{i}}_{-\infty} m_\eps(t, x) dx,
\end{multline*}
and again by positivity of $m_\eps$ we conclude that $x_t^{N+1-i} = -x_{-t}^{i}$ for all $i$ and $t$. Analogous arguments based on the fact that $m_\eps( T/4 + t) = m_\eps( T/4 - t)$ for all $t$, provide $\bx_{T/4 - t} = \bx_{T/4 + t} $ for all $t$.

\end{remark} 
 We state now the  following coerciveness property of the functional $J_T^N$. 
  \begin{proposition}\label{compactness} Let  $\bx = \bx^N \in \mathcal{K}_T^N$ 
be such that $J_T^N(\bx)\leq C$ for some $C>0$ independent of $N$.  Then, up to passing to a subsequence, 
 as $N\to +\infty$,  $m_{\bx}^N\to \mu$  in $C_{\text{per}}(\R, \mathcal{P}_p(\R))$ for all $p<2$ and   $\widetilde w_\bx^N \to \zeta \in \mathcal M([-T/2,T/2] \times \R)$ narrowly, where $m_\bx^N$ and $\widetilde w_\bx^N$ are defined in \eqref{empirical}, \eqref{empiricalspeed}. Moreover  $\mu(t, dx)=m(t)dx$, $\zeta(dt,dx)= wdt\otimes dx$,  and  $(m,w)\in \cK_T$. 
 
If moreover, $\bx\in \cK_T^{S,N}$ for all $N$, then $(m,w)\in \cK_T^S$. 
  \end{proposition} 
  \begin{proof}   
  Since $J_T^N(\bx)\leq C$, we get  that  $\frac1N \sum_{i=1}^N \int_{-T/2}^{T/2} \frac{|\dot{x}^i_t|^2}{2} dt +\frac1N \sum_{i=1}^N \int_{-T/2}^{T/2}  W(x^i_t) dt\leq C+C_K$, recalling Proposition \ref{proplimite}, where $C_K$ is a constant independent of $N$ and $\bx$.  Therefore, by Remark \ref{remreg}, $m^N(t)$ are equicontinuous, uniformly in $N$, and moreover by the properties of $W$, 
  $\int_{-T/2}^{T/2}  \int_{\R}|x|^2m^N_{\bx}(t)dxdt\leq C$, for some constant $C$ independent of $N$. Using the equicontinuity, we deduce that there exists $C_T>0$ independent of $N$ such that 
 $  \int_{\R}|x|^2m^N_{\bx} (t)dx\leq C_T$ for all $t\in [-T/2, T/2]$. So, up to passing to a subsequence, recalling Lemma \ref{equiconv} and Ascoli Arzel\`a theorem, we get that $m^N_{\bx}\to \mu$ in $C([-T/2, T/2], \cP_p(\R))$ for all $p<2$ and moreover $T$-periodicity passes to the limit. By lower semicontinuity of the potential energy recalled in Lemma \ref{lscW} and assumption \eqref{assw0}, we get that $\mu\in C(\R, \cP_2(\R))$.
 By Proposition \ref{convparticelle}, we get that $\mu$ has a density $m$, that is $\mu(t,dx)= m(t,x)dx$ and moreover $\|m\|_\infty\leq 1$. So, $\mu\in \cP_{2}^r$.

Recalling \eqref{form1}, we have that
\[ \int_{-T/2}^{T/2} \int_\R \left|\frac{dw^N_\bx(t)}{dm_{\bx}^N(t)}\right|^2 m_{\bx}^N(t)(dx)dt =\frac{1}{N}\sum_{i=1}^N\int_{-T/2}^{T/2}  |\dot{x}^i_t|^2 dt \leq C+C_K.\]
Now, arguing  as in \eqref{tv},  
that the total variation of $|\tilde w_\bx^N|$ is bounded by  $(T(C+C_K))^{1/2}$.  
So,  extracting a further subsequence, we can assume that $\tilde w_\bx^N$  converges narrowly to some
  measure $\zeta$ on  $[-T/2,T/2]\times \R$. 
Again arguing as in \eqref{tv},  for every $\phi\in C^\infty_c([-T/2, T/2]\times \R, \R)$ there holds
\begin{align*}& \int_{-T/2}^{T/2}  \int_\R \phi(t,x) \zeta (dt, dx) =\lim_N  \int_{-T/2}^{T/2} \int_\R \phi(t,x) \frac{dw^N_\bx(t)}{dm_{\bx}^N(t)} m_{\bx}^N(t)(dx)dt\\
\leq & \lim_N  \left(\int_{-T/2}^{T/2} \int_\R  \phi^2(t,x)  m_{\bx}^N(t)(dx)dt \right)^{1/2} \left( \int_{-T/2}^{T/2} \int_\R  \left|\frac{dw^N_\bx(t)}{dm_{\bx}^N(t)}\right|^2 m_{\bx}^N(t)(dx)dt\right)^{1/2} \\
\leq &  \left(\int_{-T/2}^{T/2} \int_\R  \phi^2(t,x)  m(t)dxdt \right)^{1/2} \left(C_K+C\right)^{1/2}   \end{align*} which implies that $\zeta $ is absolutely continuous with respect to $dt\otimes m dx$. So $\zeta= w dt\otimes  dx$ for some $\mu$-measurable function $w$.
 Finally, it is easy to check that $(m,w)$ satisfy the continuity equation.  \end{proof}
Finally as  a corollary of the $\Gamma$-convergence result Theorem \ref{gamma}, of the compactness result  proved in Proposition \ref{compactness}, and by the fact that minimizers of $J_T^N$ have uniformly bounded support, as proved in Theorem \ref{thmNparticle} (iii),  we get the following result. 
\begin{theorem}\label{conv} Let $\bx = \bx^N \in \mathcal{K}_T^N$ (resp.  $\bx \in \mathcal{K}_T^{N,S}$) be a minimizer of  $J_T^N$ in $\cK_T^N$ (resp. in $\cK_T^{S,N}$). Then, every limit point $(m,w)$ of $\bx$ (in the sense of Proposition \ref{compactness}) is a minimizer of $J_T$ in $\cK_T$ (resp. in $\cK_T^S$). 
 
\end{theorem}  
\begin{proof}  
By Proposition \ref{compactness}, if $(m,w)$ is a limit point of $\bx\in \cK_T^N$, then $(m,w)\in \cK_T$. 

Moreover, if $\bx$ is a minimizer of $J_T^N$, then by Theorem \ref{thmNparticle} (iii),  
$|x^i_t| \le R_0 + 1$ for all $t$ and $i = 1, \ldots N$ 
 where $R_0$ is as in \eqref{assw0}.  This implies that for all $t$  the support of $m_{\bx}^N(t)$ is contained   in $[-R_0-1, R_0+1]$. Therefore, convergence in  $C_{\text{per}}(\R, \mathcal{P}_p(\R))$ for some $p\ge1$ is equivalent to convergence in   $C_{\text{per}}(\R, \cP_2(\R))$.   Then by Proposition \ref{compactness},  $m_{\bx}^N\to \mu$  in $C_{\text{per}}(\R, \mathcal{P}_2(\R))$. 

Moreover by the $\Gamma$-liminf inequality in  Theorem \ref{gamma}, there holds that 
$\liminf_N J_T^N(\bx)\geq J_T(m,w)$. Let $(\bar m, \bar w)\in \cK_T$ be a minimizer of $J_T$, then by the $\Gamma$-limsup inequality in  Theorem \ref{gamma}, there exists $\overline\bx\in \cK_T^N$ such that $\limsup_N J_T^N(\overline\bx)\leq  J_T(\overline m, \overline w)$. 

So, we conclude, by minimality of $\bx$ and of $(\bar m, \bar w)$, that  $J_T(\overline m, \overline w)\geq \liminf_N J_T^N(\bar \bx)\geq \liminf_N J_T^N(\bx)\geq J_T(m,w)\geq J_T(\overline m, \overline w)$, which implies that $(m,w)$ is a minimizer of $J_T$. 
 
\end{proof}

  \section{Brake orbits of the mean-field problem in dimension $1$}\label{sectionfinal} 
Using the convergence of minimizers of the discrete problem to minimizers of the mean-field one, we explain in this section how to deduce some qualitative properties of brake orbits in the continuous setting. We will show in particular that  brake orbits which minimize the functional $J_T$ share the same properties as equilibria of the system: they have  compact support, independent of the period, and they are characteristic functions of appropriate intervals (with time dependent extremes).  

\begin{corollary}\label{finalcor} Let $d=1$ and assume \eqref{assK}, \eqref{pos}, \eqref{assw0}, \eqref{ref}, and \eqref{assw1}. 
Then there exists $(m_T,w_T)\in \cK_T^{S}$  minimizing $J_T$  in $\cK_T^S$ 
such that 
\[m_T(t)=\chi_{(a_T(t), a_T(t)+1)}\qquad w(t)= -\dot a_T(t)\chi_{(a_T(t), a_T(t)+1)}\] 
where $a_T:\R\to [-R_0-1, R_0]$  with $R_0$ as in \eqref{assw0}, is  a $T$- periodic $C^2$ function  which is  a minimizer of the energy 
\[x \mapsto \int_{-T/2}^{T/2} \frac{|\dot x_t|^2}{2}dt+\int_{-T/2}^{T/2}\int_{x_t}^{x_t+1} W(s)ds dt\] 
among $T$ periodic curves $x$ such that $x_{t+\frac{T}4}=x_{\frac{T}4-t}$,  and $x_t=-x_{-t}-1$.
In particular, there holds \[ a_T''(t)=W(a_T(t)+1)-W(a_T(t)),\]
and $a_T(0)=-\frac{1}{2}=a_T\left(\frac{T}{2}\right)$, $\dot a_T\left(\pm \frac{T}4\right)=0$.
Finally, \[ \lim_{T\to +\infty}  \int_{a_T(\pm T/4)}^{a_T(\pm T/4)+1} W(s)ds =0.\]  
  \end{corollary} 
\begin{proof} 
Assume that the minimizer  $(m,w)\in \cK_T^{S}$ is a limit point of $\bx = \bx^N \in \mathcal{K}_T^{N,S}$, which in turn minimizes $J_T^N$ in $\cK_T^{S,N}$ (as defined in Proposition \ref{compactness} and Theorem \ref{conv}).  By Theorem \ref{thmNparticle} (iii) and Theorem \ref{propositionsaturation}, we have $x_t^i=x_t^1+\frac{i-1}{N}$   and $x_t^1 \in [-R_0-1, R_0]$ for all $t, i$.  Moreover, since $m_{\bx}^N\to m$ narrowly, by Proposition \ref{convparticelle} and Remark \ref{rem1dconvparticelle}, $m(t)=\chi_{[a_T(t), a_T(t)+1]}$ where $a_T(t)=\lim_N x_t^1$ up to subsequence. In particular $a_T(0)= -\frac{1}{2}$ since $x_0^1=-\frac{N-1}{2N}$ by Remark \ref{remenergia}, $a$ is $T$-periodic and satisfies $a_T\left(t+\frac{T}4\right)=a_T\left(\frac{T}4-t\right)$ and $a_T(t)=-a_T(-t)-1$. 

By  Theorem \ref{conv},  recalling Proposition \ref{proplimite} and Lemma \ref{lscW}, there holds that 
\[  \lim_N \int_{-T/2}^{T/2} \int_{\R}W(x)m_{\bx}^N(t)dxdt= \int_{-T/2}^{T/2}\overline W(a_T(t))dt\quad \text{and}\quad  \lim_N\int_{-T/2}^{T/2} \mathcal{I}(m_{\bx}^N(t))dt= T\mathcal{I}(\chi_{[0,1]}) \] where $\overline W(s)=\int_s^{s+1} W(u)du$  
and moreover $\lim_N J_T^N(\bx)= J_T(m,w)$. Moreover, by Remark \ref{remenergia}, up to passing to a subsequence, we may assume that $x_t^1\to a_T(t)$ uniformly in $C^1$ as $N\to +\infty$.  Therefore, there holds
\begin{equation}\label{limxdot}
 \lim_N \int_{-T/2}^{T/2} \frac{|\dot x^1_t|^2}{2}dt=  \int_{-T/2}^{T/2} \frac{|\dot a_T(t)|^2}{2}dt= \int_{-T/2}^{T/2}\int_{\R} \frac{|w|^2}{2m}dxdt=\int_{-T/2}^{T/2}\int_{a_T(t)}^{a_T(t)+1} \frac{|w|^2}{2}dx dt.
 \end{equation}
Using the fact that $(m,w)$ is a distributional solution to the continuity equation we get, for all $T$-periodic $\phi\in C^{\infty}(\R\times \R)$ and with compact support in $x$
\begin{align*} 0=&\int_{-T/2}^{T/2}\frac{d}{dt}\int_{\R} \phi(t,x)m(t,x)dxdt=\int_{-T/2}^{T/2}\frac{d}{dt}\int_{a_T(t)}^{a_T(t)+1} \phi(t,x)dxdt\\
=&\int_{-T/2}^{T/2}\int_{a_T(t)}^{a_T(t)+1} \phi_t(t,x)dxdt +\int_{-T/2}^{T/2}\dot a_T(t) (\phi(t, a_T(t)+1)-\phi(t, a_T(t)) dt \\
=&\int_{-T/2}^{T/2}\int_{\R} w(t,x) \phi_x(t,x)dxdt +\int_{-T/2}^{T/2} \int_{\R} \dot a_T(t)  \chi_{[a_T(t), a_T(t)+1]}\phi_x(t,x) dx dt. \end{align*}
Therefore, by the fundamental lemma of calculus of variations,  we conclude that $w(t,x)=-  \dot a_T(t)  \chi_{[a_T(t), a_T(t)+1]}+c$, and by the equality in \eqref{limxdot}, we get $c=0$. 

Therefore $J_T(m,w)=  \int_{-T/2}^{T/2} \frac{|\dot a_T(t)|^2}{2}dt+\int_{-T/2}^{T/2}\overline W(a_T(t))dt-T\mathcal{I}(\chi_{[0,1]}) $  and moreover, reasoning as in Remark \ref{remenergia}  
and using the minimality of $(m,w)$, we conclude that $a_T(t)$ is a minimizer of $\int_{-T/2}^{T/2} \frac{|\dot x_t|^2}{2}+\overline W(x_t)dt$ 
 among $T$ periodic curves $x_t$ such that $x_{t+\frac{T}4}=x_{\frac{T}4-t}$,  and $x_t=-x_{-t}-1$. In particular we get that $a''_T(t)=\overline W'(a_T(t))$. 

 Finally, note that in dimension $1$, recalling assumption \eqref{assw0}, $\mathcal{M}^+=\{\chi_{(r,r+1)} \ : r\in [a^+-r_0,a^+ + r_0-1]\}$ and $\mathcal{M}^-= \{\chi_{(r,r+1)} \ : r\in [-a^+-r_0,-a^+ + r_0-1]\}$ (note that by \eqref{assw1}, $2r_0\geq 1$).  Moreover $d_2(\chi_{(a_T(t), a_T(t)+1)}, \mathcal{M}^\pm )=\inf_{ r\in [a^\pm-r_0,a^\pm + r_0-1]} |r- a_T(t)|$, then by Theorem \ref{mfgex}, we conclude that   \[ \lim_{T\to +\infty}  \int_{a_T(T/4)}^{a_T( T/4)+1} W(s)ds =0\qquad \text{and }\quad  \lim_{T\to +\infty}  \int_{a_T(-T/4)}^{a_T(- T/4)+1} W(s)ds =0.\]  
 
  \end{proof}

\appendix
\section{Some algebraic facts}\label{algapp}

Let $N \in \mathbb N$ and $\bx = (x^1, \ldots, x^N) \in \R^N$. Let
\[
d^i = d^i(\bx) := x^{i-1} - x^{i}, \qquad \text{for $i=2, \ldots, N$},
\]
and
\[
\bm^J = \bm^J(\bx) := \frac1{N-J} \sum_{i=J+1}^N x^i - \frac1{J} \sum_{i=1}^J x^i \qquad \text{for $J=1, \ldots, N-1$},
\]
which is the difference between the mean of $(x^{J+1}, \ldots, x^N)$ and the mean of $(x^1, \ldots, x^{J})$.

\begin{lemma}\label{lemA1} For all $J=1, \ldots, N-1$, the following identity holds:
\[
\bm^J = \frac{N-1}{N-J}\bm^1 - \sum_{j=1}^{J-1}\left(\frac{N-j}{N-J} - \frac jJ \right)d^{j+1}.
\]
Note that $\frac{N-j}{N-J} - \frac jJ  > 0$ for all $j \le J-1$.
\end{lemma}

\begin{proof} It is straightforward to prove by induction that for $i=2,\ldots,N$
\[
x^i = x^1 + \sum_{j=2}^i d^j,
\]
and for $J\geq 2$ 
\begin{equation}\label{J1}
\sum_{i =1}^J x^i = Jx^1 + \sum_{i =2}^J (J+1-i)d^i.
\end{equation}
Therefore,
\begin{equation}\label{J2}
\sum_{i=J+1}^N x^i = \sum_{i =1}^N x^i - \sum_{i =1}^J x^i = (N-J)x^1 + \sum_{i =2}^N (N+1-i)d^i - \sum_{i =2}^J (J+1-i)d^i .
\end{equation}
Then, using \eqref{J1} and \eqref{J2} we get
\begin{align*}
& (N-J)\bm^J - (N-1) \bm^1 = \sum_{i=J+1}^N x^i - \frac{N-J}{J} \sum_{i=1}^j x^i -\sum_{i=2}^N x^i +(N-1)x^1 \\ 
=&   (N-J)x^1 + \sum_{i =2}^N (N+1-i)d^i - \sum_{i =2}^J (J+1-i)d^i -   \frac{N-J}{J} \left(Jx^1 + \sum_{i =2}^J (J+1-i)d^i\right)\\
-& (N-1)x^1 -\sum_{i =2}^N (N+1-i)d^i  +(N-1)x^1
\\= &    - \sum_{i =2}^J (J+1-i)d^i -   \frac{N-J}{J}  \sum_{i =2}^J (J+1-i)d^i
=     - (N-J) \sum_{i =2}^J \left(\frac{J+1-i}{N-J}  +   \frac{J+1-i}{J} \right)d^i  
\\
= & - (N-J)  \sum_{i =2}^J \left( \frac{N+1-i}{N-J} - \frac{i-1}J \right) d^i=- (N-J)  \sum_{i =1}^{J-1} \left( \frac{N -i}{N-J} - \frac{i}J \right) d^{i+1}, 
\end{align*}
which gives the claimed identity after dividing by $N-J$.
\end{proof}

\begin{lemma}\label{lemA3}
Let $J=1, \ldots, N-1$ and $\alpha > 0$. Then, 
\[
\min_{\substack{\bx \in \R^N \ : \\ d^i(\bx) \ge \alpha \ \forall i}} \bm^J(\bx) = \alpha \frac N 2.
\]
Moreover, $\bm^J(\bx) = \alpha \frac N 2$ if and only if $d^i(\bx) = \alpha$  for all $i=2, \ldots, N$.
\end{lemma}

\begin{proof} Since for all $i \ge J+1$, 
\[
x^i = x^J + \sum_{j=J+1}^i d^j,
\] we get arguing as in \eqref{J1} 
\begin{align*}
\sum_{i=J+1}^N x^i = & (N-J) x^J + \sum_{i=J+1}^N \sum_{j=J+1}^i  d^j  =  (N-J) x^J + \sum_{j=J+1}^N (N+1-j)d^j.
\end{align*}
Similarly, for all $i \le J-1$, $J\ge 2$, 
\[
x^i = x^J - \sum_{j=i+1}^J d^j, \qquad \sum_{i=1}^{J-1} x^i = (J-1)x^J - \sum_{i=2}^{J} \sum_{j=i}^J d^j=  (J-1)x^J - \sum_{j=2}^{J} (j-1) d^j .
\]
Hence,
\[
 \bm^J(\bx) = \sum_{j=J+1}^N \frac{(N+1-j)}{N-J} d^j(\bx)+ \sum_{j=2}^{J}\frac{ j-1}{J} d^j(\bx),
\]
that is, $\bm^J(\bx)$ is a linear combination of $d^j(\bx)$ with positive coefficients, which implies that   $\bm^J(\bx)$ is minimized if and only if $d^j(\bx)$ are minimal, that is $d^j(\bx) = \alpha$ for all $j$. 

Assume now that $d^2(\bx) = \ldots = d^N(\bx) = \alpha$ and compute $\bm^J(\bx)$:
\[
\frac{ \bm^J(\bx) }{\alpha} = \sum_{j=J}^{N-1} \frac{(N-j)}{N-J}  + \sum_{j=1}^{J-1}\frac{ j}{J}= N- \frac{1}{N-J}  \sum_{j=1}^{N-1} j + \frac{1}{N-J}  \sum_{j=1}^{J-1} j  + \frac{J-1}{2}=\frac{N}{2}.
\]
\end{proof}

\small
\addcontentsline{toc}{section}{References}

\begin{thebibliography}{10}

\bibitem{abc}
Y.~Achdou, M.~Bardi, and M.~Cirant.
\newblock Mean field games models of segregation.
\newblock {\em Math. Models Methods Appl. Sci.}, 27(1):75--113, 2017.

\bibitem{ags}
L.~Ambrosio, N.~Gigli, and G.~Savar\'{e}.
\newblock {\em Gradient flows in metric spaces and in the space of probability
  measures}.
\newblock Lectures in Mathematics ETH Z\"{u}rich. Birkh\"{a}user Verlag, Basel,
  second edition, 2008.

\bibitem{bfrs17}
M.~Bongini, M.~Fornasier, F.~Rossi, and F.~Solombrino.
\newblock Mean-field {P}ontryagin maximum principle.
\newblock {\em J. Optim. Theory Appl.}, 175(1):1--38, 2017.

\bibitem{BC16}
A.~Briani and P.~Cardaliaguet.
\newblock Stable solutions in potential mean field game systems.
\newblock {\em NoDEA Nonlinear Differential Equations Appl.}, 25(1):25:1, 2018.

\bibitem{bptt19}
M.~Burger, R.~Pinnau, C.~Totzeck, and O.~Tse.
\newblock Mean-field optimal control and optimality conditions in the space of
  probability measures.
\newblock {\em arXiv preprint, https://arxiv.org/abs/1902.05339}, 2019.

\bibitem{cas16}
P.~Cardaliaguet, A.~R. M\'{e}sz\'{a}ros, and F.~Santambrogio.
\newblock First order mean field games with density constraints: pressure
  equals price.
\newblock {\em SIAM J. Control Optim.}, 54(5):2672--2709, 2016.

\bibitem{carr1}
J.~A. Carrillo, Y.-P. Choi, and S.~P. Perez.
\newblock A review on attractive-repulsive hydrodynamics for consensus in
  collective behavior.
\newblock {\em In Active particles. Vol. 1. Advances in theory, models, and
  applications, Model. Simul. Sci. Eng. Technol., Birkh\"auser/Springer}, pages
  259--298, 2017.

\bibitem{carr2}
J.~A. Carrillo, M.~Fornasier, G.~Toscani, and F.~Vecil.
\newblock Particle, kinetic, and hydrodynamic models of swarming.
\newblock {\em In Mathematical modeling of collective behavior in
  socio-economic and life sciences, Model. Simul. Sci. Eng. Technol.,
  Birkh\"auser Boston, Inc., Boston}, 2010.

\bibitem{ccbrake}
A.~Cesaroni and M.~Cirant.
\newblock Brake orbits and heteroclinic connections for first order mean field
  games.
\newblock {\em arXiv preprint, https://arxiv.org/abs/1912.05874}, 2019.

\bibitem{cc}
A.~Cesaroni and M.~Cirant.
\newblock Concentration of ground states in stationary mean-field games
  systems.
\newblock {\em Anal. PDE}, 12(3):737--787, 2019.

\bibitem{notebari}
A.~Cesaroni and M.~Cirant.
\newblock Introduction to variational methods for viscous ergodic mean-field
  games with local coupling.
\newblock {\em Springer INdAM series, Contemporary Research in Elliptic PDEs
  and Related Topics}, pages 221--246, 2019.

\bibitem{c19}
M.~Cirant.
\newblock On the existence of oscillating solutions in non-monotone mean-field
  games.
\newblock {\em J. Differential Equations}, 266(12):8067--8093, 2019.

\bibitem{cn}
M.~Cirant and L.~Nurbekyan.
\newblock The variational structure and time-periodic solutions for mean-field
  games systems,.
\newblock {\em Minimax Theory Appl.}, 3:227--260, 2018.

\bibitem{fs}
M.~Fischer and F.~J. Silva.
\newblock On the asymptotic nature of first order mean field games.
\newblock {\em arXiv preprint, https://arxiv.org/abs/1903.03602}, 2019.

\bibitem{flos}
M.~Fornasier, S.~Lisini, C.~Orrieri, and G.~Savar\`e.
\newblock Mean-field optimal control as {$\Gamma$}-limit of finite agent
  controls.
\newblock {\em European Journal of Applied Mathematics}, 30(6):1153--1186.,
  2019.

\bibitem{fornasier}
M.~Fornasier and F.~Solombrino.
\newblock Mean-field optimal control.
\newblock {\em ESAIM Control Optim. Calc. Var.}, 20(4):1123--1152, 2014.

\bibitem{dffr}
M.~D. Francesco, S.~Fagioli, and E.~Radici.
\newblock Deterministic particle approximation for nonlocal transport equations
  with nonlinear mobility.
\newblock {\em J. Differential Equations}, 266(5):2830 --2868, 2019.

\bibitem{dfr}
M.~D. Francesco and M.~D. Rosini.
\newblock Rigorous derivation of nonlinear scalar conservation laws from
  follow- the-leader type models via many particle limit.
\newblock {\em Arch. Ration. Mech. Anal.}, 217(3):831--871, 2015.

\bibitem{hr19}
M.~Herty and C.~Ringhofer.
\newblock Consistent mean field optimality conditions for interacting agent
  systems.
\newblock {\em Commun. Math. Sci.}, 17(4):1095--1108, 2019.

\bibitem{Lacker}
D.~Lacker.
\newblock Limit theory for controlled {M}c{K}ean- {V}lasov dynamics.
\newblock {\em SIAM J. Control Optim.}, 55(3):1641--1672, 2017.

\bibitem{LL07}
J.-M. Lasry and P.-L. Lions.
\newblock Mean field games.
\newblock {\em Jpn. J. Math.}, 2(1):229--260, 2007.

\bibitem{LLbook}
E.~H. Lieb and M.~Loss.
\newblock {\em Analysis}, volume~14 of {\em Graduate Studies in Mathematics}.
\newblock American Mathematical Society, Providence, RI, 1997.

\bibitem{Sbook}
F.~Santambrogio.
\newblock {\em Optimal transport for applied mathematicians}, volume~87 of {\em
  Progress in Nonlinear Differential Equations and their Applications}.
\newblock Birkh\"{a}user/Springer, Cham, 2015.
\newblock Calculus of variations, PDEs, and modeling.

\bibitem{sreview}
F.~Santambrogio.
\newblock Crowd motion and evolution {PDE}s under density constraints.
\newblock In {\em S{MAI} 2017---{$8^{\rm e}$} {B}iennale {F}ran\c{c}aise des
  {M}ath\'{e}matiques {A}ppliqu\'{e}es et {I}ndustrielles}, volume~64 of {\em
  ESAIM Proc. Surveys}, pages 137--157. EDP Sci., Les Ulis, 2018.

\bibitem{ymms}
H.~Yin, P.~G. Mehta, S.~P. Meyn, and U.~V. Shanbhag.
\newblock Bifurcation analysis of a heterogeneous mean-field oscillator game
  model.
\newblock In {\em Proceedings of the 50th {IEEE} Conference on Decision and
  Control and European Control Conference, {CDC-ECC} 2011}, pages 3895--3900,
  2011.

\end{thebibliography}

\medskip

\begin{flushright}
\noindent \verb"annalisa.cesaroni@unipd.it"\\
Dipartimento di Scienze Statistiche\\ Universit\`a di Padova\\
Via Battisti 241/243, 35121 Padova (Italy)

\smallskip

\noindent \verb"cirant@math.unipd.it"\\
Dipartimento di Matematica ``Tullio Levi-Civita'' \\ Universit\`a di Padova\\
Via Trieste 63, 35121 Padova (Italy)
\end{flushright}
\end{document}